\documentclass[11pt]{amsart}

\usepackage[noadjust]{cite}
\usepackage[colorlinks, citecolor=red, urlcolor=blue, bookmarks=false, hypertexnames=true]{hyperref}
\usepackage{empheq}

\theoremstyle{plain}
\newtheorem{theorem}	{Theorem}[section]
\newtheorem*{theorem*}	{Theorem \ref{thm:appl}}

\newtheorem{corollary}	[theorem]	{Corollary}
\newtheorem{lemma}	[theorem]	{Lemma}

\theoremstyle{definition}

\newtheorem{remark}       [theorem]  {Remark}

\numberwithin{equation}{section}

\DeclareMathOperator{\trace}{trace}
\DeclareMathOperator{\grad}{grad}

\allowdisplaybreaks

\begin{document}
	\title[Biharmonic hypersurfaces with three distinct principal curvatures]
	{On the biharmonic hypersurfaces with three distinct principal curvatures in space forms}
	
	\author{\c Stefan Andronic, Yu Fu and Cezar Oniciuc}
	
	\address{Faculty of Mathematics, Al. I. Cuza University of Iasi,
		Blvd. Carol I, no. 11, 700506 Iasi, Romania} \email{stefanandronic215@gmail.com}
	
	\address{School of Data Science and Artificial Intelligence, Dongbei University of Finance and
		Economics, Dalian 116025, P. R. China} \email{yufu@dufe.edu.cn}
	
	\address{Faculty of Mathematics, Al. I. Cuza University of Iasi,
		Blvd. Carol I, no. 11, 700506 Iasi, Romania} \email{oniciucc@uaic.ro}
	
	
	\begin{abstract}
		In \cite{YuFu} there was proved that any biharmonic hypersurface with at most three distinct principal curvatures in space forms has constant mean curvature. At the very last step of the proof, the argument relied on the fact that the resultant of two polynomials is a non-zero polynomial. In this paper we point out that, in fact, there is a case, and only one, when this resultant is the zero polynomial and therefore the original proof is not fully complete. Further, we prove that in this special case we still obtain that the hypersurface has constant mean curvature.
	\end{abstract}
	
	
	
	\maketitle
	
	\section{Introduction}
		
	Biharmonic maps $\varphi : M \to N$ between Riemannian manifolds are critical points of the bienergy functional and represent a natural generalization of the well-known harmonic maps. Their study was suggested in the mid-$60's$ by J. Eells and J.H. Sampson (see \cite{EellsSampson1964}, \cite{EellsSampson1966}), but the first articles where biharmonic maps were systematically studied appeared in the mid-80's (see \cite{Jiang1}, \cite{Jiang2}). In those articles, G.-Y. Jiang derived the first and the second variation formulas for the bienergy functional
	$$
	E_2:C^\infty (M,N) \to \mathbb R,\quad E_2(\varphi) = \frac 12 \int_{M} |\tau(\varphi)|^2\ v_g,
	$$ 
	where $M$ is compact and $\tau(\varphi) = \trace \nabla d\varphi$ is the tension field of $\varphi$.
	
	The Euler-Lagrange equation for the bienergy is given by the vanishing of the bitension field, i.e. 
	\begin{equation}\label{EulerLagrangeBienergy}
		\tau _2 (\varphi) = -\Delta^\varphi \tau(\varphi) - \trace R^N\left (d\varphi (\cdot), \tau(\varphi)\right ) d\varphi(\cdot) = 0,
	\end{equation}
	where $\Delta^\varphi$ is the rough Laplacian acting on the sections of $\varphi ^{-1} \left (TN\right )$ and $R^N$ is the curvature tensor field.
	
	The non-linear fourth order elliptic equation $\tau_2(\varphi) = 0$ is called the biharmonic equation. Since any harmonic map is biharmonic, we are interested in the study of the biharmonic maps which are not harmonic, called proper-biharmonic.
	
	When $\varphi: M \to N$ is an isometric immersion or, simply, when $M$ is a submanifold of $N$, we say that $M$ is biharmonic if the immersion $\varphi$ is also a biharmonic map. In this case, the biharmonic equation splits into the tangent and the normal parts, the latter being elliptic. 
	
	Independently, B.Y. Chen introduced in \cite{ChenBook1984} the notion of biharmonic submanifolds in Euclidean spaces $\mathbb R^n$, and this notion can be easily recovered from \eqref{EulerLagrangeBienergy} when the ambient manifold is flat.
	
	In spaces of non-positive sectional curvature, with only one exception (see \cite{OuTang2012}), we have only non-existence results, i.e. any biharmonic submanifold must be harmonic (minimal); for example, see \cite{FuHongZhan2021}, \cite{MontaldoOniciucRatto2016}. In particular, the following conjecture is still valid in its full generality (see \cite{Chen1991}):
	
	\textbf{Chen's conjecture.} \textit{Any biharmonic submanifold in the Euclidean space is minimal.}
	
	On the other hand, in spaces of positive sectional curvature, especially in Euclidean spheres, many examples and classification results had been obtained (see, for example, \cite{ChenOuBook2020}, \cite{FetcuOniciuc2022}, \cite{FuYangZhan}, \cite{GuanLiVrancken2021}, \cite{Nistor2022}, \cite{OniciucHabilitation}). Motivated by the known examples and results, the following conjecture has been raised (see \cite{BalmusMontaldoOniciuc}):
	
	\textbf{Conjecture ($C_1$).} \textit{Any proper-biharmonic submanifold in the Euclidean sphere has constant mean curvature.}
	
	The above conjecture was stated for submanifolds in the unit Euclidean spheres $\mathbb S^n$ because there we have examples of proper-biharmonic submanifolds, but it can be considered for submanifolds in any space form $N^n(c)$, i.e. space of constant sectional curvature $c$.
	
	In the particular case of proper-biharmonic hypersurfaces in space forms, assuming some extra hypothesis, there have been obtained several results which confirm the Conjecture ($C_1$) (see, for example, \cite{OniciucHabilitation}).
	
	One way to tackle the Conjecture ($C_1$) for hypersurfaces in space forms is to divide the study according to the number $\ell$ of distinct principal curvatures. When $\ell = 1$ everywhere, we obtain in a standard way that the hypersurface has constant mean curvature, i.e. it is CMC (see, for example, \cite{DoCarmo}). When, at any point, $\ell$ is at most $2$, ($C_1$) was proved in \cite{BalmusMontaldoOniciuc}, \cite{Dimitric1992}.
	
	When $\ell$ is at most $3$ and $m = 3$, the result was obtained in \cite{BalmusMontaldoOniciuc2010}, \cite{Defever}, \cite{HasanisVlachos1995}. Then, when $m\geq 4$, the Conjecture ($C_1$) was proved by Y. Fu in \cite{YuFu}.
	
	In our paper we show that at the very end of Y. Fu's proof there is a small gap. The author claimed that the resultant of two polynomials which appear in the proof is a non-zero polynomial. Apparently surprisingly, computing this resultant with Mathematica\textsuperscript{\textregistered}, we find that there is a case where, actually, the resultant is the zero polynomial. This occurs when the hypersurface $M^m$ in $N^{m+1}(c)$ has dimension $m=7$, the multiplicities of the three distinct principal curvatures are $1$, $3$, $3$ and $c \neq 0$. In fact, this special case has been announced by Lemma \ref{LemmaBicons}. In this situation, when the resultant is the zero polynomial, we do not obtain the desired contradiction and thus the proof in \cite{YuFu} is not complete.
	
	 Further, we split the analysis as follows. When $c=0$, we prove that any proper-biharmonic hypersurface with at most three distinct principal curvatures is CMC, for any $m$. In fact, the $c=0$ case was already proved in \cite{YuFuTohoku2015} in a similar way, but we keep it here for the sake of completeness. When $c\neq0$, we first modify the polynomials and we get a new resultant. We show that the only case when the new resultant is the zero polynomial is the one mentioned above, no matter if $c<0$, $c=0$ or $c>0$. We end the paper proving that even in this special case we can conclude that the hypersurface $M^7$ is CMC.
	 
	 The Conjecture ($C_1$) is an important issue for biharmonic hypersurfaces in spheres because it would imply that any proper-biharmonic hypersurface in Euclidean spheres has constant mean curvature and constant scalar curvature. This fact is related to the following version of the Chern's Conjecture:\\ 
	 \textbf{Generalized Chern's Conjecture.} \textit{Any hypersurface with constant mean and scalar curvatures in the Euclidean sphere is isoparametric.}
	 
	 Further, if Conjecture ($C_1$) and the Generalized Chern's Conjecture are proved then, using the classification of proper-biharmonic isoparametric hyperspheres in spheres obtained in \cite{IchiyamaInoguchiUrakawa2009}, \cite{IchiyamaInoguchiUrakawa2010}, we can reach the full classification of proper-biharmonic hypersurfaces in Euclidean spheres, as it was conjectured in \cite{BalmusMontaldoOniciuc}: \\
	 \textbf{Conjecture ($C_2$).} \textit{Let $M^m$ be a proper-biharmonic hypersurface in $\mathbb S^{m+1}$. Then $M$ is either an open part of the small hypersphere $\mathbb S^m(1/\sqrt 2)$ of radius $1 / \sqrt 2$ or an open part of $\mathbb S^{m_1} (1/\sqrt 2) \times \mathbb S^{m_2} (1/\sqrt 2)$, $m_1 + m_2 = m$, $m_1 \neq m_2$.}
	 
	 We mention that, when $m\geq 4$ and $M$ has three distinct principal curvatures, the Generalized Chern's Conjecture was proved in \cite{AlmeidaBritoScherfnerWeiss2020}. Therefore, as the Generalized Chern's Conjecture is of local nature, Conjecture ($C_2$) is proved for $m\geq 4$ and $M$ with at most three distinct principal curvatures (see Corollary \ref{CorollaryClassificationHypersurfacesBiharm3Curvatures}).
	
	\section{Conventions}\label{Conventions}
	In this paper, all manifolds are assumed to be connected and oriented. In general, the metrics are indicated by $\langle \cdot, \cdot \rangle$ or, simply, not explicitly mentioned. The Levi-Civita connection of the Riemannian manifold $M$ is denoted by $\nabla$.
	
	The rough Laplacian defined on the set of all sections in the pull-back bundle $\varphi^{-1} \left (TN\right )$ is given by 
	$$
	\Delta ^\varphi = -\trace \left ( \nabla ^\varphi \nabla ^\varphi - \nabla ^\varphi _\nabla\right )
	$$
	and the curvature tensor field on $N$ is 
	$$
	R^N\left ( U, V\right ) W = \left [ \nabla ^N _U, \nabla ^N _V\right ] W - \nabla ^N _{[U, V]} W.
	$$
	For a hypersurface $M^m$ in $N^{m+1}$ we denote mean curvature function by $f =\left (\trace A \right ) / m$, where $A = A_\eta$ is the shape operator of $M$ and $\eta$ is a unit section in the normal bundle.
	
	\section{Preliminaries}\label{Preliminaries}
	
	We briefly recall that when $M$ is a hypersurface in a space form we have the following characterization of the biharmonicity (for $c=0$ see \cite{ChenBook1984} and for any $c$ see \cite{CaddeoMontaldoOniciuc2001}, \cite{ChenJH1993}, \cite{Jiang1}).
	\begin{theorem}
		Let $M^m$ be a hypersurface in a space form $N^{m+1} (c)$. Then $M$ is biharmonic if and only if 
		\begin{equation}\label{BiharmonicEquationHypersurfaces}
			\left \{
			\begin{array}{rl}
				\rm {(i)}  & \Delta f + \left ( |A|^2 - mc \right ) f = 0\\
				\rm {(ii)} & 2 A\left (\grad f \right ) + mf\grad f = 0
			\end{array}
			\right . .
		\end{equation}
	\end{theorem}
	We recall the result obtained by Y. Fu concerning the proper-biharmonic hypersurfaces with three distinct principal curvatures in space forms. 
	\begin{theorem}[\cite{YuFu}]\label{MainTheorem}
		Let $M^m$ be a proper-biharmonic hypersurface in $N^{m+1}(c)$, $m\geq 4$, with at most three distinct principal curvatures. Then $M^m$ is CMC.
	\end{theorem}
	
	In order to give a direct application of Theorem \ref{MainTheorem}, we first recall the recent result obtained in \cite{AlmeidaBritoScherfnerWeiss2020}.
	
	\begin{theorem}[\cite{AlmeidaBritoScherfnerWeiss2020}]
		Let $M^m$ be a hypersurface with constant mean and scalar curvatures in $\mathbb S^{m+1}$, $m\geq 4$. Assume that $M$ has three distinct principal curvatures at any point. Then $M$ is isoparametric.
	\end{theorem}

	\begin{remark}
		The above result is a generalization to the non-compact case of that obtained in \cite{Chang1994}.
	\end{remark}
	
	Now, using the classification of proper-biharmonic isoparametric hypersurfaces obtained in \cite{IchiyamaInoguchiUrakawa2009}, \cite{IchiyamaInoguchiUrakawa2010} we can conclude
	
	\begin{corollary}\label{CorollaryClassificationHypersurfacesBiharm3Curvatures}
		Let $M^m$ be a proper-biharmonic hypersurface in $\mathbb S^{m+1}$, $m\geq 4$, with at most three distinct principal curvatures. Then $M$ is either an open part of the small hypersphere $\mathbb S^m(1/\sqrt 2)$ of radius $1 / \sqrt 2$ or an open part of $\mathbb S^{m_1} (1/\sqrt 2) \times \mathbb S^{m_2} (1/\sqrt 2)$, $m_1 + m_2 = m$, $m_1 \neq m_2$.
	\end{corollary}

	\noindent {Before giving a slightly different proof of Theorem \ref{MainTheorem}, we will recall the fundamental equations of hypersufaces in space forms.}
	\begin{itemize}
		\item The Gauss Equation:
		\begin{equation}\label{GaussEquation}
			R(X, Y) Z = c \left ( \left \langle Y, Z \right \rangle X - \left \langle X, Z \right \rangle Y \right ) + \left \langle A(Y), Z \right \rangle A(X) - \left \langle A(X), Z \right \rangle A(Y),
		\end{equation} 
		for any $X, Y, Z \in C(TM)$.
		
		\item The Codazzi Equation: 
		\begin{equation} \label{CodazziEquation}
			\left ( \nabla _X A\right ) (Y) = \left ( \nabla _Y A \right ) (X),
		\end{equation}
		for any $X, Y \in C(TM)$.
	\end{itemize}
	
	\section{Proof of the theorem \ref{MainTheorem}}
		\noindent {For the sake of completeness, we will present here a detailed proof of the result. Basically, excepting the very last part, it coincides with the original proof but it is organized slightly different.}
		
		For an arbitrary hypersurface $\varphi : M^m \to N^{m+1} (c)$, we denote by 
		$$
		\overline k_1 \geq \overline k_2 \geq \ldots \geq \overline k_m
		$$ 
		its principal curvatures. The functions $\{ \overline k_i\}_{i\in \overline {1,m}}$ are continuous on $M$, for any $i\in \overline {1,m}$, but not necessarily smooth everywhere. The set of all points at which the number of distinct principal curvatures is locally constant is an open and dense subset of $M$. We will denote by $M_A$ this set. On a non-empty connected component of $M_A$, which is open in $M_A$, thus in $M$, the number of distinct principal curvatures is constant. Therefore, on that connected component, the multiplicities of the distinct principal curvatures are constant and so the $\overline k_i$'s are smooth and $A$ is smoothly locally diagonalizable (see \cite{Nomizu}, \cite{RyanTohoku}, \cite{RyanOsaka}).
		
		We will show that $\grad f = 0$ on every connected component of $M_A$ and thus, from density, $\grad f = 0$ on $M$, i.e. $f$ is constant.
		
		We choose an arbitrary connected component of $M_A$. Because $M$ has at most three distinct principal curvatures, on this component we have: either each of its points is umbilical, or each of its points has exactly two distinct principal curvatures, or each of its points has exactly three distinct principal curvatures. For simplicity, we will denote by $M$ the chosen connected component.
		
		If $M$ is umbilical or if $M$ has exactly two distinct principal curvatures at any point, then the result is already proved (see \cite{BalmusMontaldoOniciuc}).
		
		We suppose now that $M$ has exactly three distinct principal curvatures at any point. In this case, $A$ is (locally) diagonalizable with respect to an orthonormal frame field $\{ E_1, \ldots, E_m\}$, thus $A(E_i) = \overline k_i E_i$, for any $i \in \overline {1,m}$.
		
		Assume, by way of contradiction, that $\grad f \neq 0$ and, at the end of the proof, we will get a contradiction. If necessary, we can restrict ourselves to an open subset, also denoted (for simplicity) by $M$, and we can assume that $\grad f \neq 0$ at any point of $M$. Using a similar argument, we can suppose that $f = |H| > 0$ on $M$.
		
		We will denote by $\mathcal D$ the distribution orthogonal to that determined by $\grad f$. It is known that $\mathcal D$ is completely integrable (see \cite{HasanisVlachos1995}, \cite{Nistor}).
				
		Next, for any $p \in M$ we will denote by
		\begin{align*}
			k_1(p) &= \overline {k}_1 (p) = \ldots = \overline {k}_{m_1} (p)\\
			k_2(p) &= \overline {k}_{m_1+1} (p) = \ldots = \overline {k}_{m_1+m_2} (p)\\
			k_3(p) &= \overline {k}_{m_1+m_2+1} (p) = \ldots = \overline {k}_{m} (p)
		\end{align*}
		the distinct principal curvatures at the point $p$, $m_1 + m_2 + m_3 = m$. The multiplicities $m_1$, $m_2$, $m_3$ are constant functions on $M$.
		
		We know that $A(\grad f) = -\frac {m} {2} f \grad f$, thus we can suppose that 
		$$
		k_1 = -\frac {m} {2} f \quad \text{and} \quad E_1 = \frac {\grad f} {|\grad f|}
		$$
		on $M$.
		
		In the first part of the proof we will try to get as much information as we can only from the tangent part of the biharmonic equation.
		
		We will prove that $E_1(k_1) \neq 0$ at any point of $M$.
		$$
		E_1(k_1) = -\frac {m} {2} E_1(f) = -\frac {m} {2} |\grad f| \neq 0
		$$
		at any point of $M$.
		
		We mention that all the following formulas hold on $M$, if it is not stated otherwise.
		
		From the definition of $E_1$ we also get that for any $\alpha \in \overline {2, m}$, 
		$$
		E_\alpha (f) = 0.
		$$
		We define $\omega ^k_j : C(TM) \to \mathbb {R}$ such that $\nabla _X E_j = \omega ^k_ j (X) E_k$. It is easy to prove that $\omega ^k_j$ is a one-form and that it has the property $\omega ^j_i = - \omega ^i_j$, for any $ i, j \in \overline {1,m}$.
		
		From Codazzi equation \eqref{CodazziEquation} we obtain for any $i, j \in \overline {1, m}$ 
		$$
		E_i(\overline {k}_j ) E_j + \sum _{\ell =1} ^m (\overline k_j - \overline k_\ell) \omega ^\ell_j (E_i) E_\ell = E_j(\overline k_i) E_i + \sum _{\ell =1} ^m (\overline k_i - \overline k_\ell) \omega ^\ell _i (E_j)E_\ell.
		$$
		Considering the fact that $\left \{ E_i \right \}_{i \in \overline {1, m}}$ is an orthonormal frame field, we get
		
		\begin{align}
			&E_i(\overline k_j) = ( \overline k _i - \overline k _j) \omega ^j_i (E_j),\label{EquationCodazzi1}\\
			&(\overline k _j - \overline k _\ell) \omega ^\ell _j (E_i) = (\overline k_i - \overline k_\ell) \omega ^\ell _i (E_j),\label{EquationCodazzi2}
		\end{align}
		for any mutually distinct $i, j, \ell \in \overline {1, m}$. In these relations we do not use the Einstein summation convention.
		
		Next, we will show that the multiplicity of $k_1$ is $m_1 = 1$. We suppose that $m_1 \geq 2$. Let $\alpha \neq 1$ be such that $\overline k_\alpha = k_1$. In (\ref{EquationCodazzi1}), for $i = 1$ and $j = \alpha$, we have
		$$
		E_1(\overline k_\alpha) = (\overline k_1 - \overline k_\alpha) \omega ^j_i (E_j)
		$$
		which is equivalent to $E_1(k_1) = 0$, contradiction. Therefore, $m_1 = 1$.
		
		We set $r = 1+m_2$, so we have 
		\begin{equation*} 
			k_2 (p) = \overline k_2 (p) = \ldots = \overline k_r (p)
		\end{equation*}
		and
		\begin{equation*}
			k_3 (p) = \overline k_{r+1} (p) = \ldots = \overline k_m (p),	
		\end{equation*}
		for $(m+1) / 2 \leq r < m$. Thus, the multiplicities of $k_2$ and $k_3$ are $r-1$ and $m-r$, respectively.
		
		From the definition of the mean curvature function we have
		\begin{align*}
			mf =& \trace A = -\frac m 2 f + (r-1)k_2 + (m-r)k_3, 
		\end{align*}
		thus 
		\begin{equation}\label{CurvaturesDependency}
			k_3 = \frac {3m} {2(m-r)} f - \frac {r-1} {m-r} k_2.
		\end{equation}
		Using the fact that $k_2 \neq k_1$, $k_3 \neq k_1$ and $k_3 \neq k_2$, we obtain that 
		\begin{equation}\label{CurvaturesNEQ}
			k_2 \neq - \frac m 2 f,\quad k_2 \neq \frac {m(m-r+3)} {2(r-1)} f\quad \text{and} \quad k_2 \neq \frac {3m} {2(m-1)} f
		\end{equation}
		at any point of $M$.
		
		Our goal is to find two polynomial equations in the variables $f$ and $k_2$ which have common solutions. Therefore, their resultant must vanish. It turns out that the resultant is a polynomial equation in $f$ and, by continuity, $f$ must be constant which contradicts $\grad f \neq 0$.
		
		We will use the tangent part of the biharmonic equation to obtain the properties of the connection forms.
		\begin{lemma}
			The connection forms $\omega _i^j$ have the following properties:
			\begin{equation}\label{EquationOmega1}
				\omega ^1_j (E_i) = \omega ^1_i(E_j),\quad \forall i, j \in \overline {2, m},
			\end{equation}
			\begin{equation}\label{EquationOmega2}
				\omega ^1_i (E_1) = 0,\quad \forall i \in \overline {2, m},
			\end{equation}
			\begin{equation}\label{Equation3}
				 E_i(E_1(f)) = 0,\quad \forall i \in \overline {2, m},
			\end{equation}
			\begin{equation}\label{EquationOmega3} 
				\omega ^1_j (E_i) = 0,\quad \forall i, j \in \overline {2, m},\ i\neq j ,
			\end{equation}
			\begin{equation}\label{EquationOmega5}
				\omega ^i_1 (E_i) = \frac {E_1(\overline k _i)} {k_1 - \overline k_i},\quad \forall i \in \overline {2, m} ,
			\end{equation}
			\begin{equation}\label{EquationOmega7}
				\omega^2_m (E_2) = \frac {E_m(k_2)} {k_3 - k_2}.
			\end{equation} 
		\end{lemma}
		\begin{proof}
			For any $i, j \in \overline {2, m}$ we have
			$$
			[E_i, E_j] (f) = E_i(E_j(f)) - E_j(E_i(f)) = 0.
			$$
			On the other hand, 
			\begin{align*}
				[E_i, E_j](f) &= \left ( \nabla _{E_i} E_j - \nabla _{E_j} E_i\right ) (f)\\
				&= \left ( \omega ^\ell _j (E_i) - \omega ^\ell _i (E_j)\right ) E_\ell (f)\\
				&= \left ( \omega ^1_j (E_i) - \omega ^1_i(E_j)\right) E_1(f).
			\end{align*}
			The fact that $E_1(f) \neq 0$ at any point of $M$ implies (\ref{EquationOmega1}).
			
			For $j=1$ and $i\neq 1$ in (\ref{EquationCodazzi1}), we obtain $E_i (k_1) = (\overline k_i - k_1) \omega ^1 _i (E_1)$. Because $E_i(k_1) = 0$, for any $ i \in \overline {2,m}$, we obtain (\ref{EquationOmega2}).
			
			We have
			\begin{align*}
				[E_1, E_i](f) &= \left ( \nabla _{E_1} E_i - \nabla _{E_i} E_1\right ) (f)\\
				&= \left ( \omega ^\ell _i (E_1) - \omega ^\ell _1 (E_i)\right ) E_\ell (f)\\
				&= \left ( \omega ^1 _i (E_1) - \omega ^1_1(E_i)\right ) E_1(f)\\
				&= \omega ^1_i (E_1) E_1(f).
			\end{align*}
			Using (\ref{EquationOmega2}), it is easy to prove (\ref{Equation3}).
			
			For $\ell = 1$, $j \neq 1$, $i \neq 1$ and $i\neq j$ in (\ref{EquationCodazzi2}) we have
			$$
			(\overline k_j - k_1) \omega ^1_j (E_i) = (\overline k _i - k_1) \omega ^1_i (E_j)
			$$
			and, together with (\ref{EquationOmega1}), it implies that 
			
			$$
			(\overline k _j - \overline k_i ) \omega ^1_j(E_i) = 0,
			$$
			which proves (\ref{EquationOmega3}). 
			
			From (\ref{EquationCodazzi1}), we get that 
			$$
			E_1(\overline k_i) = (k_1 - \overline k _i) \omega ^i_1(E_i)
			$$
			which is equivalent to (\ref{EquationOmega5}).
			
			From (\ref{EquationCodazzi1}) we get (\ref{EquationOmega7}).
		\end{proof}
	
		Using both the normal and the tangent parts of the biharmonic equation, we show that the function $k_2$ is constant along the leaves of $\mathcal D$.
		
		\begin{lemma}\label{LemmaLeavesD}
			We have $E_i(k_2) = 0$, for any $i \in \overline {2,m}$.
		\end{lemma}
		\begin{proof} 
			Because $m\geq 4$ and $(m+1) / 2 \leq r$, we obtain $r-1 \geq 2$. From (\ref{EquationCodazzi1}) for any $i, j \in \overline {2,r}$, we get 
			\begin{equation}\label{Equation1}
				E_i(k_2) = 0.
			\end{equation}
			\textbf{Case 1:} If $m-r \geq 2$, then from (\ref{EquationCodazzi1}), for any $ i, j \in \overline {r+1, m}$, we have
			\begin{align*}
				E_i(k_3) = 0 \Leftrightarrow& E_i \left ( \frac {3m} {2(m-r)} f - \frac {r-1} {m-r} k_2\right ) = 0\\
				\Leftrightarrow & \frac {3m} {2(m-r)} E_i(f) - \frac {r-1} {m-r} E_i(k_2) = 0.	
			\end{align*}
			Therefore $E_i(k_2) = 0 \text{ on } M,$ for any $ i \in \overline {2,m}$.
			
			\textbf{Case 2:} If $m-r = 1$.
			
			Using (\ref{Equation1}) we obtain that 
			\begin{equation}\label{Equation2}
				E_i(k_2) = 0,\quad \forall i \in \overline {2, m-1}.
			\end{equation}
			We have 
			\begin{align*}
				[E_i, E_1] (k_2) &= E_i(E_1(k_2)) - E_1(E_i(k_2)) = E_i(E_1(k_2)),\ \forall i \in \overline {2, m-1}.
			\end{align*}
			On the other hand,
			\begin{align*}
				[E_i, E_1] (k_2) &= \left ( \nabla _{E_i} E_1 - \nabla _{E_1} E_i\right ) (k_2)\\
				&= \left ( \omega ^\ell _1(E_i) - \omega ^\ell _i (E_1)\right ) E_\ell (k_2)\\
				&= \left ( \omega ^1_1 (E_1) - \omega ^1_i(E_1)\right ) E_1(k_2) + \left ( \omega ^m_1(E_i) - \omega ^m_i(E_1)\right ) E_m(k_2)\\
				&= -\omega ^m_i (E_1) E_m(k_2).
			\end{align*}
			For $\ell = m$, $j=1$ and $i \in \overline {2, m-1}$ in (\ref{EquationCodazzi1}) we have
			$$
			(k_1 - k_3) \omega ^m_1(E_i) = (k_2 - k_3) \omega ^m_i(E_1),
			$$
			which implies $\omega ^m_i (E_1) = 0.$
			
			Thus, 
			\begin{equation}\label{Equation4}
				E_i(E_1(k_2)) = 0,\quad \forall i \in \overline {2, m-1}.
			\end{equation}
			For $i = m$ and $j = 2$ in (\ref{EquationCodazzi2}) we obtain
			$$
			(k_2 - \overline k_\ell) \omega ^\ell _2 (E_m) = ( k_3 - \overline k_\ell) \omega ^\ell _m(E_2),\quad \forall \ell \in \{ 1, 3,\ldots, m-1\}.
			$$
			This implies that
			$$
			\omega ^\ell _m (E_2) = 0,\quad \forall \ell \in \overline {3, m-1}.
			$$
			We know from \eqref{EquationOmega3} that $\omega ^1_m (E_2) = \omega ^m_m(E_2) = 0$, thus
			\begin{equation}\label{EquationOmega6}
				\omega ^\ell_m(E_2) = 0,\quad \forall \ell \in \{1, 3, \ldots, m\}.
			\end{equation}
			From Gauss equation \eqref{GaussEquation} we get that 
			\begin{align}
				R(E_2, E_m) E_1 =& c\left ( \left \langle E_m, E_1 \right \rangle E_2 - \left \langle E_2, E_1 \right \rangle E_m \right ) + \left \langle A(E_m), E_1 \right \rangle A(E_2) - \label{EquationCurvature1}\\
								 & - \left \langle A(E_2), E_1 \right \rangle E_m\notag\\
								=& 0. \notag
			\end{align}
			We will compute the curvature using the definition, i.e. 
			$$
			R(E_2, E_m) E_1 = \nabla _{E_2} \nabla _{E_m} E_1 - \nabla _{E_m} \nabla _{E_2} E_1 - \nabla _{[E_2, E_m]} E_1.
			$$
			Using \eqref{EquationOmega3}, \eqref{EquationOmega5}, \eqref{EquationOmega7}, \eqref{Equation4}, \eqref{EquationOmega6}  we obtain
			\begin{align*}
				\nabla_{E_2} \nabla _{E_m} E_1 =& \nabla _{E_2} \left ( \omega ^\ell _1 (E_m) E_\ell \right )\\
				=& E_2\left(\omega ^\ell _1 (E_m)\right)	E_\ell + \omega ^\ell _1 (E_m) \nabla _{E_2} E_\ell\\
				=& E_2\left ( \omega ^1_1 (E_m)\right ) E_1  + E_2\left ( \omega ^m _1 (E_m)\right ) E_m + \omega ^\ell _1 (E_m) \omega ^j _\ell (E_2) E_j\\
				=& E_2 \left ( \frac {E_1(k_3)} {k_1 - k_3} \right ) E_m + \omega ^1 _1 (E_m) \omega ^j_1 (E_2) E_j + \omega ^m_1(E_m) \omega ^j_m (E_2) E_j \\
				=& E_2 \left ( \frac {E_1(k_3)} {k_1 - k_3} \right) E_m + \frac {E_1(k_3)} {k_1 - k_3} \omega ^2_m (E_2) E_2.
			\end{align*}
			From (\ref{CurvaturesDependency}), (\ref{Equation3}) and (\ref{Equation4}), we have
			$$
			E_i(E_1(k_3)) = 0,\quad \forall i \in \overline {2, m-1}.
			$$
			Therefore, 
			$$
			\nabla _{E_2} \nabla _{E_m} E_1 = \frac {E_1(k_3)} {k_1 - k_3} \omega ^2_m(E_2) E_2 = \frac {E_1(k_3) E_m(k_2)} {(k_1 - k_3) (k_3-k_2)} E_2.
			$$
			Now we will compute the second term from the curvature formula.
			\begin{align*}
				\nabla _{E_m} \nabla _{E_2} E_1 =& \nabla _{E_m} \left ( \omega ^\ell _1 (E_2) E_\ell\right )\\
				=& E_m\left ( \omega ^\ell _1 (E_2)\right ) E_\ell + \omega ^\ell _1 (E_2) \nabla _{E_m} E_\ell \\
				=& E_m \left ( \omega ^\ell _1 (E_2)\right ) E_\ell + \omega ^\ell _1 (E_2) \omega ^j _\ell (E_m) E_j\\
				=& E_m \left ( \omega ^1 _1 (E_2)\right ) E_1 + E_m \left ( \omega ^2_1(E_2)\right ) E_2 + \omega ^1_1 (E_2) \omega ^j_1(E_m) E_j + \\
				& + \sum _{j = 3} ^m \omega ^2_1 (E_2) \omega ^j _2 (E_m) E_j\\
				=& E_m \left ( \frac {E_1 (k_2)} {k_1 - k_2} \right ) E_2 + \frac {E_1 (k_2)} {k_1 - k_2} \sum _{j=3} ^m \omega ^j _2 (E_m) E_j.
			\end{align*}
			For the final term we have
			\begin{align*}
				\nabla _{[E_2, E_m]} E_1 =& \nabla _{\left ( \nabla _{E_2} E_m - \nabla _{E_m} E_2\right )} E_1\\
				=& \nabla _{\left ( \omega ^\ell _m (E_2) - \omega ^\ell _2 (E_m) \right ) E_\ell} E_1\\
				=& \omega ^\ell _m (E_2) \nabla _{E_\ell} E_1 - \omega ^\ell _2 (E_m) \nabla _{E_\ell} E_1\\
				=& \omega ^\ell _m (E_2) \omega ^j_1(E_\ell) E_j - \omega ^\ell _2 (E_m) \omega ^j_1 (E_\ell) E_j\\	 
				=& \omega ^2_m (E_2) \omega ^j _1(E_2) E_j - \omega ^\ell _2 (E_m) \omega ^1_1 (E_\ell) E_1 -\\
				& - \omega ^1 _2(E_m) \omega ^j_1(E_1) E_j - \sum _{\ell =2} ^m \omega ^\ell _2 (E_m) \omega ^\ell _1 (E_\ell) E_\ell\\
				=& \frac {E_m(k_2) E_1(k_2)} {(k_3 - k_2) (k_1 -k_2)} E_2 - \frac {E_1(k_2)} {k_1 - k_2} \sum _{\ell = 3} ^{m-1} \omega ^\ell _2 (E_m) E_\ell - \frac {E_1(k_3)} {k_1 - k_3} \omega ^m_2 (E_m) E_m.
			\end{align*}
			Therefore, (\ref{EquationCurvature1}) is equivalent to
			\begin{align*}
				& \nabla _{E_2} \nabla _{E_m} E_1 - \nabla _{E_m} \nabla _{E_2} E_1 - \nabla _{[E_2, E_m]} E_1 = 0\\
				\Leftrightarrow & \frac {E_1(k_3) E_m(k_2)} {(k_1 - k_3) (k_3 - k_2)} E_2 - E_m\left ( \frac {E_1(k_2)} {k_1 -k_2} \right ) E_2 - \frac {E_1(k_2)} {k_1 - k_2} \sum _{\ell = 3} ^m \omega ^\ell _2 (E_m) E_\ell -\\
				& - \frac {E_m(k_2) E_1 (k_2)} {(k_3 - k_2) (k_1 - k_2)} E_2 + \frac {E_1(k_2)} {k_1 - k_2} \sum _{\ell = 3} ^{m-1} \omega ^\ell _2 (E_m) E_\ell + \frac {E_1(k_3)} {k_1 -k_3} \omega ^m _2 (E_m) E_m = 0\\
				\Leftrightarrow & \left ( \frac {E_1(k_3) E_m(k_2)} {(k_1 - k_3) (k_3 - k_2)} - E_m \left ( \frac {E_1(k_2)} {k_1 - k_2}\right ) - \frac {E_m(k_2) E_1(k_2)} {(k_3 -k_2) (k_1 - k_2)} \right ) E_2 +\\
				& + \left ( \frac {E_1(k_3)} {k_1 - k_3} - \frac {E_1 (k_2)} {k_1 - k_2} \right ) \omega ^m _2 (E_m) E_m = 0.
			\end{align*}
			The fact that $E_2$ and $E_m$ are linearly independent implies
			\begin{equation}\label{RelationDer1}
				E_m \left ( \frac {E_1 (k_2)} {k_1 - k_2} \right ) = \frac {1} {k_3 - k_2} \left ( \frac {E_1(k_3)} {k_1 - k_3} - \frac {E_1(k_2)} {k_1 - k_2}\right ) E_m(k_2).
			\end{equation}
			Taking into account that 
			$$
			|A|^2 = \sum _{i=1} ^m \left \langle AE_i, AE_i \right \rangle = \sum _{i=1} ^m \overline k_i^2 = k_1^2 + (m-2)k_2^2 + k_3 ^2
			$$
			and
			\begin{align*}
				\Delta f =& - E_1(E_1(f)) + \left (\nabla _{E_1} E_1 \right ) f - \sum _{i=2} ^{m-1} \left ( E_i(E_i(f)) - \left ( \nabla _{E_i} E_i \right ) f \right ) -\\
				& -E_m (E_m(f)) + \left ( \nabla _{E_m} E_m \right ) f\\
				=& - E_1(E_1(f)) - \left ( \frac {(m-2)E_1(k_2)} {k_1 - k_2} + \frac {E_1(k_3)} {k_1 - k_3}\right ) E_1 (f),
			\end{align*}
			replacing in \eqref{BiharmonicEquationHypersurfaces}(i) we get
			\begin{equation}\label{RelationDer2}
				-E_1(E_1(f)) - \left ( \frac {(m-2) E_1(k_2)} {k_1 - k_2} + \frac {E_1(k_3)} {k_1 - k_3} \right ) E_1(f) + (k_1 ^2 + (m-2)k_2^2 + k_3^2) f = mcf.
			\end{equation}
			We will compute $E_m(E_1(E_1(f)))$. We have
			\begin{align*}
				[E_1, E_m] (E_1(f)) =& \left ( \nabla _{E_1} E_m - \nabla _{E_m} E_1 \right ) (E_1(f)) \\
				=& \left ( \omega ^j_m(E_1) - \omega ^j _1 (E_m)\right ) E_j(E_1(f)) \\
				=& \left ( \omega ^1_m (E_1) - \omega ^1_1 (E_m)\right ) E_1(E_1(f)) = 0.
			\end{align*}
			On the other hand
			$$
			[E_1, E_m] (E_1(f)) = E_1(E_m(E_1(f))) - E_m(E_1(E_1(f))) = - E_m(E_1(E_1(f))).
			$$
			Thus,
			\begin{equation}\label{Equation5}
				E_m(E_1(E_1(f))) = 0.
			\end{equation}
			Differentiating \eqref{RelationDer2} by $E_m$ we obtain
			\begin{align*}
				& -E_m(E_1(E_1(f))) - E_m \left ( \frac {(m-2)E_1(k_2)} {k_1 - k_2} + \frac {E_1(k_3)} {k_1 - k_3}\right ) E_1 (f) - \\
				& -\left ( \frac {(m-2) E_1(k_2) } {k_1 - k_2} + \frac {E_1(k_3)} {k_1 - k_3}\right ) E_m (E_1 (f)) + \\
				& + 2(k_1 E_m(k_1) + (m-2)k_2 E_m(k_2) + k_3 E_m(k_3)) f + \\
				& + (k_1 ^2 + (m-2) k_2 ^2 + k_3^2 ) E_m(f) = mc E_m(f),
			\end{align*}
			which using \eqref{Equation5} leads to
			\begin{align}
				& - E_m \left (\frac {(m-2)E_1(k_2)} {k_1 -k_2}\right )E_1(f) - E_m\left (\frac {E_1(k_3)} {k_1 - k_3} \right ) E_1 (f) + \label{RelationDer3}\\
				& + 2 ((m-2) k_2 E_m(k_2) + k_3 E_m (k_3))f = 0.\notag
			\end{align}
			Further, we will compute each term of \eqref{RelationDer3} separately.
			
			Using formula (\ref{RelationDer1}) we obtain
			\begin{align*}
				& - E_m \left ( \frac {(m-2) E_1(k_2)} {k_1 - k_2}\right ) =\\
				=& - (m-2) \left ( \frac {E_1(k_3)E_m(k_2)} {(k_1 - k_3) (k_3 - k_2)} - \frac {E_m(k_2) E_1(k_2)} {(k_3 - k_2) (k_1 - k_2)} \right )\\
				=& - \frac {m-2} {k_3 - k_2} \left ( \frac {E_1(k_3)} {k_1 -k_3} - \frac {E_1 (k_2)} {k_1 - k_2} \right ) E_m(k_2).
			\end{align*}
			It is easy to see that 
			$$
			E_m\left (\frac {E_1(k_2)} {k_1 -k_2}\right ) = \frac {E_m(E_1(k_2))} {k_1 - k_2} + \frac {E_1(k_2) E_m(k_2)} {(k_1 - k_2) ^2}.
			$$
			From (\ref{RelationDer1}) we get
			$$
			\frac {E_m(E_1(k_2))} {k_1 - k_2} + \frac {E_1(k_2)E_m(k_2)} {(k_1 - k_2)^2} = \frac {E_1(k_3) E_m(k_2)} {(k_1 - k_3) (k_3 - k_2)} -   \frac {E_m(k_2)E_1(k_2)} {(k_1 - k_2) (k_3 - k_2)},
			$$
			which is equivalent to
			\begin{align*}
				E_m(E_1(k_2)) =& \left ( - \frac {E_1(k_2)} {k_1 - k_2} + \frac {k_1 - k_2} {k_3 - k_2} \frac {E_1(k_3)} {k_1 - k_3} -\frac {k_1 - k_2} {k_3 - k_2} \frac {E_1(k_2)} {k_1 - k_2}\right ) E_m(k_2)	\\
				=& \left ( \frac {k_1 - k_2} {k_3 - k_2} \frac {E_1(k_3)} {k_1 - k_3} - \frac {k_1 - 2k_2 + k_3} {k_3 - k_2} \frac {E_1 (k_2)} {k_1 - k_2}\right ) E_m(k_2).	
			\end{align*}
			Now, we can compute
			\begin{align*}
				E_m\left ( \frac {E_1(k_3)} {k_1 - k_3}\right ) =& \frac {E_m(E_1(k_3))}{k_1 - k_3} - \frac {E_1(k_3)\left ( E_m(k_1) - E_m(k_3)\right )} {(k_1 - k_3)^2}\\
				=& \frac {E_m(E_1(k_3))} {k_1 - k_3} + \frac {E_1(k_3) E_m(k_3)} {(k_1 - k_3)^2}\\
				=& \frac {E_m\left (E_1\left (\frac 3 2 mf - (m-2) k_2\right )\right )} {k_1 - k_3} + \frac {E_1(k_3) E_m\left ( \frac 32 mf - (m-2)k_2\right )} {(k_1 - k_3)^2}\\
				=& - \frac {(m-2) E_m(E_1(k_2))} {k_1 - k_3} - \frac {(m-2) E_1(k_3) E_m(k_2)} {(k_1 - k_3) ^2}\\
				=& -(m-2) \frac {k_1 -2k_2 + k_3} {k_3 - k_2} \left ( \frac {E_1(k_3)} {k_1 - k_3} - \frac {E_1(k_2)} {k_1 - k_2}\right ) \frac {E_m(k_2)} {k_1 - k_3}.
			\end{align*}
			The last term of (\ref{RelationDer3}) is
			\begin{align*}
				& 2((m-2)k_2 E_m(k_2) + k_3 E_m(k_3))f =\\
				=& 2\left ( (m-2) k_2 E_m(k_2) - (m-2)\left ( \frac 3 2 mf - (m-2)k_2\right )E_m(k_2)\right ) f\\
				=& (m-2) \left ( -3mf + 2(m-1)k_2\right ) fE_m(k_2).
			\end{align*}
			Therefore, relation (\ref{RelationDer3}) becomes
			\begin{equation*}
				\frac 2 {k_1 - k_3} \left ( \frac {E_1(k_3)} {k_1 - k_3} - \frac {E_1(k_2)} {k_1 - k_2} \right ) E_1(f) E_m(k_2) + (-3mf + 2(m-1)k_2)f E_m(k_2) = 0.
			\end{equation*}
			Assuming that $E_m(k_2) \neq 0$ at any point of an open subset $U$ of $M$. On $U$ we get
			\begin{equation}\label{RelationDer4}
				\frac 2 {k_1 - k_3} \left ( \frac {E_1(k_3)} {k_1 - k_3} - \frac {E_1(k_2)} {k_1 - k_2} \right ) E_1(f) + (-3mf + 2(m-1)k_2)f = 0.
			\end{equation}
			Differentiating (\ref{RelationDer4}) by $E_m$ we obtain
			\begin{align*}
				& - \frac {2\left ( E_m(k_1) - E_m(k_3)\right )} {(k_1 - k_3)^2} \left ( \frac {E_1(k_3)} {k_1 - k_3} - \frac {E_1(k_2)} {k_1 - k_2}\right ) E_1(f) +\\
				& + \frac 2 {k_1 - k_3} \left ( -(m-2)\frac {k_1 - 2k_2 + k_3} {(k_3 - k_2) (k_1 - k_3)} \left ( \frac {E_1 (k_3)} {k_1 - k_3} - \frac {E_1 (k_2)} {k_1 - k_2} \right ) E_m(k_2) -\right.\\
				& \left . - \frac 1 {k_3 - k_2} \left ( \frac {E_1(k_3)} {k_1 - k_3} - \frac {E_1 (k_2) } {k_1 - k_2} \right ) E_m(k_2)\right ) E_1 (f) + \\
				& + (-3mE_m(f) + 2(m-1)E_m(k_2))f = 0,
			\end{align*}
			which is equivalent to 
			\begin{align*}
				&  - \frac {2\left ((m-1)k_1 - 3(m-2)k_2 + (2m-5)k_3\right )} {(k_1 - k_3)(k_3 - k_2)} \left ( \frac {E_1(k_2)} {k_1 - k_3} - \frac {E_1(k_2)} {k_1 - k_2}\right ) E_1(f) E_m(k_2) +\\
				& + 2(m-1)f E_m(k_2) = 0.
			\end{align*}
			Using (\ref{CurvaturesDependency}) we obtain
			$$
			(m-1)k_1 - 3(m-2)k_2 + (2m - 5)k_3 = \frac 12 m(5m-14)f - 2(m-2)(m-1)k_2
			$$
			and
			$$
			\frac {k_1 - k_3} 2 = \frac {-2mf + (m-2)k_2} 2.
			$$
			Therefore
			\begin{align}\label{RelationDer5}
				& \frac {\frac 12 m (14 - 5m)f + 2(m-1)(m-2)k_2} {(k_1 - k_3) (k_3 - k_2)} \left ( \frac {E_1(k_3)} {k_1 - k_3} - \frac {E_1(k_2)} {k_1 - k_2} \right ) E_1 (f) +\\
				& + (m-1) \left ( -2mf + (m-2)k_2\right ) f = 0. \notag
			\end{align}
			Using (\ref{CurvaturesDependency}) and (\ref{RelationDer4}), relation (\ref{RelationDer5}) becomes
			$$
			(m-2) \left ( -9mf + 6(m-1)k_2\right ) f = 0,
			$$	
			which implies that 
			$$
			k_2 = \frac {3mf} {2(m-1)}
			$$
			on $U$, which contradicts (\ref{CurvaturesNEQ}).
			
			From \textbf{Case 1} and \textbf{Case 2} we obtain the conclusion.
		\end{proof} 
		
		\begin{remark}
			The distinct principal curvatures $k_1$, $k_2$ and $k_3$ are constant along the leaves of $\mathcal D$.
		\end{remark}
		
		Using the fact that the function $k_2$ is constant along the leaves of $\mathcal D$ we get more information about the connection forms.
		\begin{lemma}\label{LemmaConnectionForms}
			The connection forms $\omega ^i_j$ satisfy
			\begin{itemize}
				\item for any $i \in \overline {2, m}$:
				\begin{equation}
					\omega ^1_i (E_1) = 0, \label{EquationOmega8} 
				\end{equation}
				\item for any $ j \in \overline {2, m}$:
				\begin{equation}
					\omega ^j_1 (E_j) = \frac {E_1(\overline k_j)} {k_1 - \overline k_j}, \label{EquationOmega9}
				\end{equation}
				\item for any $i \in \overline {r+1, m}$, $j \in \overline {2, r}$:
				\begin{equation}
					\omega ^j _i (E_j) = 0, \label{EquationOmega10}
				\end{equation}
				\item for any $i \in \overline {2, r}$ and $j \in \overline {r+1, m}$:
				\begin{equation}
					\omega ^j_i(E_j) = 0, \label{EquationOmega11}
				\end{equation}
			\end{itemize}
			\begin{itemize}
				\item for any $j, \ell \in \overline {2, r}$, $\ell \neq j$:
				\begin{equation}
					\omega^\ell _ 1 (E_j) = 0, \label{EquationOmega12}
				\end{equation}
				\item for any $j, \ell \in \overline {r+1, m}$, $j\neq \ell$:
				\begin{equation}
					\omega ^\ell _1 (E_j) = 0, \label{EquationOmega13}
				\end{equation}
				\item for any $i \in \overline {2, r}$, $j, \ell \in \overline {r+1, m}$, $j\neq \ell$:
				\begin{equation}
					\omega ^\ell _i (E_j) = 0, \label{EquationOmega14}
				\end{equation}
				\item for any $i, \ell \in \overline {2, r}$, $j \in \overline {r+1, m}$, $i\neq \ell$:
				\begin{equation}
					\omega^\ell_j (E_i) = 0, \label{EquationOmega15}
				\end{equation}
				\item for any $i \in \overline {2, r}$, $j \in \overline {r+1, m}$:
				\begin{equation}
					\omega^1_j(E_i) = 0 \label{EquationOmega16}
				\end{equation}
				and
				\begin{equation}
					\omega^1_i(E_j) = 0, \label{EquationOmega17}
				\end{equation}
				\item for any $j \in \overline {2, r}$, $\ell \in \overline {r+1, m}$:
				\begin{equation}
					\omega ^\ell _j (E_1) = 0. \label{EquationOmega18}
				\end{equation}
			\end{itemize}
		\end{lemma}
		\begin{proof}
			For formulas (\ref{EquationOmega8})$-$(\ref{EquationOmega11}) we use (\ref{EquationCodazzi1}) and for relations (\ref{EquationOmega12})$-$(\ref{EquationOmega18}) we use (\ref{EquationCodazzi2}).
		\end{proof}
		
		From the previous lemmas we can write the Levi-Civita connection of $M$ as follows.
		\begin{lemma}
			The following relations hold
			
			\begin{itemize}
				\item The integral curves of $E_1$ are geodesic, i.e.
				\begin{equation}\label{RelationConnection1}
					\nabla _{E_1} E_1 = 0.
				\end{equation}
				\item For any $i \in \overline {2,m}$,
				\begin{equation}\label{RelationConnection2}
						\nabla _{E_i} E_1 = \frac {E_1(\overline k _i)} {k_1 - \overline k_i} E_i.
				\end{equation}
				\item For any $j \in \overline {2, r}$ and $i \in \overline {r+1, m}$,
				\begin{equation}\label{RelationConnection3}
					\nabla _{E_i} E_j = \sum _{\substack {\ell = 2 \\ \ell \neq j}} ^r \omega ^\ell _j (E_i) E_\ell.
				\end{equation}
				\item For any $i, j \in \overline {2, r}$, $i\neq j$,
				\begin{equation}\label{RelationConnection4}
					\nabla _{E_i} E_j = \sum _{\substack{\ell = 2 \\ \ell \neq j}} ^r \omega ^\ell _j (E_i) E_\ell.
				\end{equation}
				\item For any $j \in \overline {2, r}$,
				\begin{equation}\label{RelationConnection5}
					\nabla _{E_1} E_j = \sum _{\substack {\ell = 2 \\ \ell \neq j}} ^r \omega ^\ell _j (E_1) E_\ell.
				\end{equation}
				\item For any $i, j \in \overline {r+1, m}$, $i\neq j$,
				\begin{equation}\label{RelationConnection6}
					\nabla _{E_i} E_j = \sum _{\substack{\ell = r+1\\ \ell \neq j}} ^m \omega^\ell _j (E_i)E_\ell.
				\end{equation}
				\item For any $j\in \overline {r+1, m}$, $i\in \overline {2, r}$,
				\begin{equation}\label{RelationConnection7}
					\nabla _{E_i} E_j = \sum _{\ell = r+1} ^m \omega^\ell _j (E_i) E_\ell.
				\end{equation}
				\item For any $j \in \overline {r+1, m}$,
				\begin{equation}\label{RelationConnection8}
					\nabla _{E_1} E_j = \sum _{\substack {\ell = r+1 \\ \ell \neq j}} ^m \omega ^\ell _j (E_1) E_\ell.
				\end{equation}
				\item For any $i \in \overline {2, r}$,
				\begin{equation}\label{RelationConnection9}
					\nabla _{E_i} E_i = - \frac {E_1(k_2)} {k_1 - k_2} E_1 + \sum _{\substack {\ell = 2\\\ell \neq i}} ^r \omega ^\ell _i (E_i) E_\ell.
				\end{equation}
				\item For any $i \in \overline {r+1, m}$,
				\begin{equation}\label{RelationConnection10}
					\nabla _{E_i} E_i = - \frac {E_1(k_3)} {k_1 - k_3} E_1 + \sum _{\substack{\ell = r+1 \\ \ell \neq i}} ^m \omega ^\ell _i(E_i) E_\ell.
				\end{equation}
			\end{itemize}
		\end{lemma}
		\begin{proof}
			 This lemma is proved directly using Lemma \ref{LemmaConnectionForms}.
		\end{proof}
	
		We set 
		\begin{equation}\label{DefinitionOmegaTheta}
			\Omega = \frac {E_1 (k_2)} {k_1 - k_2}\quad \text{and} \quad \Theta = \frac {E_1(k_3)} {k_1 - k_3}.
		\end{equation}
		We will use the Gauss equation and the tangent part of the biharmonic equation to infer some relation that $\Omega$ and $\Theta$ must satisfy.
		\begin{lemma}
			The following relations hold:
			\begin{equation}\label{EquationGauss1}
				E_1(\Omega) + \Omega^2 = -c - k_1 k_2.
			\end{equation}
			\begin{equation}\label{EquationGauss2}
				E_1(\Theta ) + \Theta^2 = -c - k_1 k_3.
			\end{equation}
			\begin{equation}\label{EquationGauss3}
				\Omega\Theta = -c - k_2 k_3.
			\end{equation}
			\begin{equation}\label{EquationGauss4}
				E_1(E_1(k_2)) + 2\Omega E_1(k_2) - \Omega E_1(k_1) + (k_1k_2 + c) (k_1 - k_2) = 0.
			\end{equation}
			\begin{equation}\label{EquationGauss5}
				E_1(E_1(k_3)) + 2 \Theta E_1(k_3) - \Theta E_1(k_1) + (k_1 k_3 + c) (k_1 - k_3) = 0.	
			\end{equation}
		\end{lemma}
		\begin{proof}
			The first three relations are obtained from \eqref{GaussEquation}.
			\begin{itemize}
				\item For $X = E_1$, $Y = E_2$ and $Z = E_1$ we obtain
				\begin{align*}
					R(E_1 , E_2) E_1 =& \nabla _{E_1} \nabla _{E_2} E_1 - \nabla _{E_2} \nabla _{E_1} E_1 - \nabla _{[E_1, E_2]} E_1\\
					=& \nabla _{E_1} \left (\Omega \cdot E_2\right ) - \nabla _{\left (\nabla _{E_1} E_2 - \nabla _{E_2} E_1\right )} E_1\\
					=& E_1(\Omega) E_2 + \Omega \cdot \sum _{\ell = 3} ^r \omega ^\ell _2 (E_1) E_\ell - \nabla _{\left (\sum \limits _{\ell = 3} ^r \left (\omega ^\ell _2 (E_1) E_\ell\right ) - \Omega \cdot E_2 \right )} E_1\\
					=& E_1(\Omega) E_2 + \Omega \cdot \sum _{\ell = 3} ^r \omega ^\ell _2(E_1) E_\ell - \sum _{\ell = 3} ^r \omega ^\ell _2 (E_1) \nabla _{E_2} E_1 + \Omega \cdot \nabla _{E_2} E_1\\
					=& \left ( E_1(\Omega) + \Omega^2\right ) E_2 + \Omega \cdot \sum _{\ell = 3} ^r \omega ^\ell _2 (E_1) E_\ell - \sum _{\ell = 3} ^r \Omega \cdot \omega ^\ell _2 (E_1) E_\ell \\
					=& \left ( E_1(\Omega) + \Omega^2\right ) E_2.
				\end{align*}
				On the other hand, from \eqref{GaussEquation}
				\begin{align*}
					R(E_1, E_2) E_1 =& c\left (\left \langle E_2, E_1 \right \rangle E_1 - \left \langle E_1, E_1 \right \rangle E_2 \right )+\left \langle A(E_2), E_1\right \rangle A(E_1) -\\
									 & - \left \langle A(E_1), E_1\right \rangle A(E_2)\\
									=& (-c - k_1k_2) E_2.
				\end{align*}
				Therefore, we get (\ref{EquationGauss1}).
				\item For $X = E_1$, $Y = E_m$ and $Z = E_1$ we have
				\begin{align*}
					R( E_1, E_m) E_1 =& \nabla _{E_1} \nabla _{E_m} E_1 - \nabla _{E_m} \nabla _{E_1} E_1 - \nabla _{[E_1, E_m]} E_1 \\
					=& \nabla _{E_1} \left (\Theta \cdot E_m\right ) - \nabla _{\left (\nabla _{E_1} E_m - \nabla _{E_m} E_1\right )} E_1\\
					=& E_1(\Theta) E_m + \Theta \cdot \nabla _{E_1} E_m - \nabla _{\left (\sum \limits _{\ell = r+1} ^{m-1} \left ( \omega ^\ell _m (E_1) E_\ell\right ) - \Theta \cdot E_m\right )} E_1\\
					=& E_1(\Theta) E_m + \Theta \cdot \sum _{\ell = r+1} ^{m-1} \omega ^\ell _m (E_1) E_\ell - \sum _{\ell = r+1} ^{m-1} \omega ^\ell _m (E_1) \nabla _{E_ \ell} E_1 + \Theta \cdot \nabla _{E_m} E_1\\
					=&\left ( E_1(\Theta) + \Theta ^2\right ) E_m.
				\end{align*}
				On the other hand, from Gauss equation \eqref{GaussEquation}
				\begin{align*}
					R(E_1, E_m) E_1 =& (-c - k_1k_3) E_m.
				\end{align*}
				Therefore, we obtain (\ref{EquationGauss2}).
				\item For $X = E_m$, $Y = E_2$ and $Z = E_m$ we have
				\begin{align*}
					R(E_m, E_2) E_m =& \nabla _{E_m} \nabla _{E_2} E_m - \nabla _{E_2} \nabla _{E_m} E_m - \nabla _{[E_m, E_2]} E_m\\
					=& \nabla _{E_m} \left ( \sum _{\ell = r+1} ^{m-1} \omega ^\ell _m (E_2) E_\ell\right ) - \nabla _{E_2} \left ( -\Theta \cdot E_1 + \sum _{\ell = r+1} ^{m-1} \omega ^\ell _m (E_m) E_\ell\right ) - \\
					& -\nabla _{\left (\nabla _{E_m} E_2 - \nabla _{E_2} E_m \right )} E_m\\
					=& \sum _{\ell = r+1} ^{m-1} \left ( E_m\left (\omega ^\ell _m(E_2) \right ) E_\ell + \omega ^\ell _m(E_2) \nabla _{E_m} E_\ell\right )+ E_2(\Theta)E_1 + \\
					& + \Theta \cdot \nabla _{E_2} E_1 - \sum _{\ell = r+1} ^{m-1} \left ( E_2 \left ( \omega ^\ell _m(E_m) \right ) E_\ell + \omega ^\ell _m (E_m) \nabla _{E_2} E_\ell\right ) -\\
					& - \nabla _{\left (\sum \limits _{\ell = 3} ^r \omega ^\ell _2(E_m) E_\ell - \sum \limits _{\ell = r+1} ^{m-1} \omega ^\ell _m (E_2) E_\ell \right )} E_m \\
					=& E_2(\Theta)E_1 + \Omega \cdot \Theta \cdot E_2 + \sum _{\ell = r+1} ^{m-1} \left ( E_m\left (\omega ^\ell _m(E_2)\right ) E_\ell + \omega ^\ell _m(E_2) \nabla _{E_m} E_\ell \right )-\\
					& - \sum _{\ell = r+1} ^{m-1} \left ( E_2 \left (\omega ^\ell _m (E_m) \right ) E_\ell + \omega ^\ell _m(E_m) \nabla _{E_2} E_\ell \right ) -\\
					& - \sum _{\ell = 3} ^r \omega ^\ell _2 (E_m) \nabla _{E_\ell} E_m - \sum _{\ell = r+1} ^{m-1} \omega ^\ell _m (E_2) \nabla _{E_\ell} E_m.
				\end{align*}
				Thus,
				$$
				\left \langle R(E_m, E_2) E_m , E_2\right \rangle = \Omega \cdot \Theta.
				$$
				On the other hand, form \eqref{GaussEquation}
				\begin{align*}
					R(E_m, E_2) E_m = (-c - k_2k_3) E_2
				\end{align*}
				and 
				$$
					\left \langle R(E_m, E_2) E_m, E_2 \right \rangle = -c - k_2k_3.
				$$
				Therefore, (\ref{EquationGauss3}) is true.
			\end{itemize}
			We have
			\begin{align*}
				E_1(\Omega) =& E_1 \left ( \frac {E_1(k_2)} {k_1 - k_2} \right ) = \frac {E_1(E_1(k_2))} {k_1 - k_2} - \frac {E_1(k_2)} {k_1 - k_2} \frac {E_1(k_1) - E_1(k_2)} {k_1 - k_2}\\
				=& \frac {E_1(E_1(k_2))} {k_1 - k_2} - \Omega \cdot \frac {E_1(k_1) - E_1 (k_2)} {k_1 - k_2}.
			\end{align*}
			The formula (\ref{EquationGauss1}) becomes (\ref{EquationGauss4}).
			
			Similarly, formula (\ref{EquationGauss2}) becomes (\ref{EquationGauss5}).
		\end{proof}
	
		Next, we will use the normal part of the biharmonic equation to get another relation concerning $\Omega$ and $\Theta$.
		\begin{lemma}\label{LemmaNormalPart}
			The following relation holds
			\begin{equation}\label{EquationNormalPart}
				-E_1(E_1(f)) - \left ( (r-1) \Omega + (m-r)\Theta \right ) E_1(f) + \left (k_1 ^2 + (r-1) k_2 ^2 + (m-r) k_3^2\right ) f = mcf.
			\end{equation}
		\end{lemma}
		\begin{proof}
			Since $M$ is a biharmonic submanifold, we have \eqref{BiharmonicEquationHypersurfaces}(i).
			
			Now, we will compute $|A|^2$ and $\Delta f$. We have
			\begin{align*}
				|A|^2 =& k_1 ^2 + (r-1)k_2^2 + (m-r)k_3^2,\\
				\Delta f =& - \left ( E_1(E_1(f)) - \left ( \nabla _{E_1} E_1\right ) f\right ) - \sum _{i = 2} ^r \left ( E_i(E_i(f)) - \left ( \nabla _{E_i} E_i \right ) f \right ) -\\
				& - \sum _{i = r+1} ^m \left ( E_i(E_i(f)) - \left ( \nabla _{E_i} E_i\right ) f \right )\\
				=& - E_1(E_1(f)) - \left ( (r-1) \Omega + (m-r) \Theta \right ) E_1(f).
			\end{align*}
			Substituting in the first relation of this proof we obtain (\ref{EquationNormalPart}).
		\end{proof}
	
		\begin{remark}
			We know that the normal part of the biharmonic equation is used in the proof of Lemma \ref{LemmaLeavesD} when $m-r=1$ and in obtaining relation \eqref{EquationNormalPart}.
		\end{remark}
	
		It is easy to see that from (\ref{CurvaturesDependency}) follows
		\begin{align}
			&E_1(k_2) = \frac {3m} {2(r-1)} E_1(f) - \frac {m-r} {r-1} \Theta \cdot (k_1 - k_3),\label{RelationDer6}\\
			&E_1(k_3) = \frac {3m} {2(m-r)} E_1(f) - \frac {r-1} {m-r} \Omega \cdot (k_1 - k_2).\label{RelationDer7}
		\end{align}
		Relation (\ref{RelationDer7}) is equivalent to
		\begin{align*}
			\Theta \cdot (k_1 - k_3) = \frac {3m} {2(m-r)} E_1(f) - \frac {r-1} {m-r} \Omega \cdot (k_1 - k_2),
		\end{align*}
		thus 
		\begin{equation}\label{RelationDer8}
			E_1(f) = \left ( - \frac {r-1} 3 f - \frac {2(r-1)} {3m} k_2 \right ) \Omega + \left ( - \frac {m-r+3} 3 f + \frac {2(r-1)} {3m} k_2 \right ) \Theta.
		\end{equation}
	
		As a consequence of \eqref{EquationGauss4}, \eqref{EquationGauss5} and \eqref{EquationNormalPart} we have
		\begin{lemma}
			The following formula holds
			\begin{align}
				& ((4-r)\Omega + (r-m+3) \Theta ) E_1(f) + \frac {3m^2(m-r+6)} {4(m-r)} f^3 - \label{Relation1}\\
				& - \frac {3m(m+4r -2)} {2(m-r)} f^2 k_2 + \frac {3m(r-1)} {m-r} fk_2^2 - 3(m+1) cf = 0\notag 
			\end{align}
		\end{lemma}
		\begin{proof}
			Using (\ref{CurvaturesDependency}), (\ref{EquationGauss4}), (\ref{EquationNormalPart}), (\ref{RelationDer6}), (\ref{RelationDer7}) relation (\ref{EquationGauss5}) yields
			\begin{align*}
				&\frac {3m} {2(m-r)} E_1(E_1(f)) - \frac {r-1} {m-r} E_1(E_1(k_2)) + \\
				& + 2 \Theta \cdot \left ( \frac {3m} {2(m-r)} E_1(f) - \frac {r-1} {m-r} E_1(k_2)\right ) + \frac m 2 \Theta \cdot E_1(f) +\\
				& + \left ( -\frac m2 f \left ( \frac {3m} {2(m-r)} f - \frac {r-1} {m-r} k_2\right ) + c \right ) \left ( -\frac m2 f - \frac {3m} {2(m-r)} f + \frac {r-1} {m-r} k_2\right ) = 0 \\ 
				\Leftrightarrow & \frac {3m} {2(m-r)} \left ( - ((r-1)\Omega + (m-r)\Theta) E_1(f) + (k_1 ^2 + (r-1) k_2 ^2 +(m-r)k_3^2 -mc) f \right ) +\\
				& + \frac {r-1} {m-r} \left ( 2 \Omega \cdot E_1(k_2) - \Omega \cdot E_1 (k_1) + (k_1k_2 + c) (k_1 - k_2) \right ) + \\
				& + \frac {3m} {m-r} \Theta \cdot E_1(f) - \frac {2(r-1)} {m-r} \Omega \cdot \Theta \cdot (k_1 - k_2) + \frac m2 \Theta \cdot E_1(f)+\\
				& + \left ( - \frac {3m^2} {4(m-r)} f^2 + \frac {m(r-1)} {2(m-r)} fk_2 + c \right ) \left ( -\frac {m(m-r+3)} {2(m-r)} f + \frac {r-1} {m-r} k_2 \right ) = 0\\ 
				\Leftrightarrow & \left ( - \frac {3m(r-1)} {2(m-r)} \Omega - \frac {3m(m-r)} {2(m-r)} \Theta \right ) E_1(f) + \\
				& + \frac {3m} {2(m-r)} \left ( \frac {m^2(m-r+9)} {4(m-r)} f^2 - \frac {3m(r-1)} {m-r} fk_2 + \frac {(r-1) (m-1)} {m-r} k_2^2 -mc \right ) f +\\
				& + \frac {r-1} {m-r} \left ( \frac {3m} {r-1} \Omega \cdot E_1(f) - \frac {2(m-r)} {r-1} \Omega \cdot \Theta \cdot (k_1 - k_3) + \frac m2 \Omega \cdot E_1(f)\right ) +\\
				& + \frac {r-1} {m-r} \left (- \frac m2 fk_2 + c\right ) \left ( - \frac m2 f - k_2 \right ) + \frac {3m} {m-r} \Theta \cdot E_1(f)+\\
				& + \frac {2(r-1)} {m-r} \left ( c + \frac {3m} {2(m-r)} fk_2 - \frac {r-1} {m-r} k_2^2\right )\left ( - \frac m2 f - k_2\right ) + \frac m2 \Theta \cdot E_1(f) +\\
				& + \frac {3m^3(m-r+3)} {8(m-r)^2} f^3 - \frac {3m^2(r-1)} {4(m-r)^2} f^2k_2 - \frac {m^2(r-1) (m-r+3)} {4(m-r)^2} f^2k_2 +\\
				& + \frac {m(r-1)^2} {2(m-r)^2} fk_2^2 - \frac {m(m-r+3)} {2(m-r)} cf + \frac {r-1} {m-r} ck_2 = 0 \\
				\Leftrightarrow & \frac m {m-r} ((4-r) \Omega + (r-m+3) \Theta ) E_1(f) + \frac {3m^3(m-r+9)} {8(m-r)^2} f^3 - \\
				& - \frac {9m^2(r-1)} {2(m-r)^2} f^2k_2 + \frac {3m(m-1)(r-1)} {2(m-r)^2} fk_2^2 - \frac {3m^2} {2(m-r)} cf - \\
				& - 2\Omega \cdot \Theta \cdot \left ( - \frac {m(m-r+3)} {2(m-r)} f + \frac {r-1} {m-r} k_2 \right ) + \frac {m^2(r-1)} {4(m-r)} f^2k_2 + \\
				& + \frac {m(r-1)} {2(m-r)} fk_2^2 - \frac {m(r-1)} {2(m-r)} cf - \frac {r-1} {m-r} ck_2 - \frac {m(r-1)} {m-r} cf - \\
				& - \frac {2(r-1)} {m-r} ck_2 - \frac {3m^2(r-1)} {2(m-r)^2} f^2k_2 + \frac {m(r-1)(r-4)} {(m-r)^2} fk_2^2 + \frac {2(r-1)^2} {(m-r)^2} k_2^3 + \\
				& + \frac {3m^3 (m-r+3)} {8(m-r)^2} f^3 - \frac {m^2(r-1)(m-r+6)} {4(m-r)^2} f^2k_2 + \frac {m(r-1)^2} {2(m-r)^2} fk_2^2 -\\
				& - \frac {m(m-r+3)} {2(m-r)} cf + \frac {r-1} {m-r} ck_2 = 0.
			\end{align*}
			Using (\ref{EquationGauss3}) we get \eqref{Relation1}.
		\end{proof}
	
		We note that \eqref{Relation1} vanishes identically when $m=7$, $r=4$ and $\delta = 0$, where 
		$$\delta = 28 k_2^2 - 98 f k_2 + 147 f^2 - 32 c.$$
		In this special case, \eqref{EquationNormalPart} is just a consequence of \eqref{EquationGauss4} and \eqref{EquationGauss5} and therefore we do not need to use the normal part of the biharmonic equation in Lemma \ref{LemmaNormalPart}. Also, the relations derived from \eqref{Relation1} cannot provide new information. This special case will appear naturally in our analysis at the end of the proof.
		
		We recall here that a submanifold for which the tangent part of the biharmonic equation vanishes is called biconservative. Since $\delta = 0$ is equivalent to $14 |A|^2 - 245 f^2 - 96 c = 0$, we can state
		
		\begin{lemma}\label{LemmaBicons}
			Let $M^7$ be a biconservative hypersurface in $N^8(c)$. Assume that $M$ has three distinct principal curvatures of multiplicities $m_1 = 1$, $m_2 = 3$, $m_3 = 3$ and $\grad f \neq 0$ at any point of $M$. Then the following relation 
			$$14 |A|^2 - 245 f^2 - 96 c = 0$$
			cannot hold on any open subset of $M$.
		\end{lemma}
	
		The proof of Lemma \ref{LemmaBicons} will be given at the end of this section.
	
		Using both the tangent and the normal parts of the biharmonic equation, we will derive more properties of the functions $\Omega$ and $\Theta$.
		
		\begin{lemma}\label{LemmaPropertiesOmegaTheta}
			The functions $\Omega$ and $\Theta$ satisfy
			\begin{align}
				& (r-1)(4-r)(mf +2k_2) \Omega ^2 + (r-m+3)(m(m-r+3) f - 2 (r-1)k_2) \Theta ^2 = \label{Relation2}\\
				=& \frac {9m^3(m-r+6)} {4(m-r)} f^3 + \frac {3m^2(r-1)(2r-2m-15)} {2(m-r)} f^2k_2 + \notag\\
				& + \frac {m(r-1)(m+11r-12 +2mr-2r^2)}{m-r} fk_2^2 + \frac {2(r-1)^2(m-2r+1)} {m-r}k_2^3- \notag\\
				& - m(2mr+4m-2r^2+5r)cf - 2(r-1)(m-2r+1)ck_2  \notag
			\end{align}
			and
			\begin{align}
				& \left ( \frac 9 4 m^3(3m-2r+17) f^3 + \frac 3 2 m^2(6r^2-43r+37+11m-11mr)f^2k_2+\right . \label{Relation6}\\
				& + m(r-1)(26r+4mr+1-4m)fk_2^2 + m(m-r)(8r-5mr-13m-17)cf - \notag\\
				& - \left . 2(r-1)^2(7+2m)k_2^3+2(m-r)(r-1)(m+17)ck_2 \right ) \Omega +\notag\\
				& + \left ( \frac 9 2 m^3(2r-2m-3)f^3 + \frac 9 2 m^2(7r-m+3-m^2+3mr-2r^2)f^2k_2 + \right . \notag\\
				& +  2m(r-1)(4m-13r-18-2mr+2m^2)fk_2^2 + \notag\\
				& +  m(m-r)(5mr-5m^2-7m-8r+42)cf + \notag\\
				& + \left . 2(r-1)^2(7+2m)k_2^3 - 2(r-1)(m-r)(m+17)ck_2 \right ) \Theta = 0. \notag
			\end{align}
		\end{lemma}
		\begin{proof}
			From \eqref{RelationDer8} and \eqref{Relation1} we get
			\begin{align*}
				& \left ( \frac {(r-1) (r-4)} 3 f + \frac {2(r-1) (r-4)} {3m} k_2 \right ) \Omega ^2 + \\
				& + \left ( \frac {2mr - 5m -2r^2+5r - 9} 3 f + \frac {2 (r-1) (m-2r+1)} {3m} k_2\right ) \Omega \cdot \Theta +\\
				& + \left ( \frac {(m-r+3)(m-r-3)} 3 f + \frac {2(r-1) (r-m+3)} {3m} k_2 \right ) \Theta ^2 + \\
				& + \frac {3m^2(m-r+6)} {4(m-r)} f^3 - \frac {3m(m+4r -2)} {2(m-r)} f^2k_2 + \frac {3m(r-1)} {m-r} fk_2^2 - 3(m+1)cf = 0.
			\end{align*}
			Using (\ref{EquationGauss3}) we obtain (\ref{Relation2}).
			
			Multiplying (\ref{Relation1}) by $\Omega$ we get
			\begin{align}
				(4-r)\Omega ^2 E_1(f) =& (r-m+3) (c+k_2k_3) E_1(f) - \left ( \frac {3m^2(m-r+6)} {4(m-r)} f^3 -\right . \label{Relation3}\\
				& \left . - \frac {3m(m+4r-2)} {2(m-r)} f^2k_2 + \frac {3m(r-1)} {m-r} fk_2 ^2 - 3(m+1)cf\right ) \Omega \notag
			\end{align}
			and multiplying (\ref{Relation1}) by $\Theta$ we obtain
			\begin{align}
				(r-m+3) \Theta ^2E_1(f) =& (4-r) (c+k_2k_3) E_1(f) - \left ( \frac {3m^2(m-r+6)} {4(m-r)} f^3 - \right . \label{Relation4}\\
				& \left . - \frac {3m(m+4r-2)} {2(m-r)} f^2k_2 + \frac {3m(r-1)} {m-r} fk_2^2 - 3(m+1) cf\right ) \Theta. \notag
			\end{align}
			Differentiating (\ref{Relation1}) by $E_1$ we get
			\begin{align*}
				&((4-r) E_1(\Omega) + (r-m+3) E_1(\Theta) ) E_1(f) + ((4-r) \Omega + (r-m+3) \Theta ) E_1(E_1(f)) +\\
				& + \frac {9m^2(m - r +6)} {4(m-r)} f^2E_1(f) - \frac {3m(m+4r-2)} {m-r} fk_2 E_1(f) - \\
				& - \frac {3m(m+4r-2)} {2 (m-r)} f^2 E_1(k_2) + \frac {3m(r-1)} {m-r} E_1(f) k_2^2 +\\
				& + \frac {6m(r-1)} {m-r} fk_2E_1(k_2) - 3(m+1)c E_1(f) = 0 \\
				\Leftrightarrow & \left ( (4-r) (-\Omega ^2 - c - k_1k_2) + (r-m+3) (-\Theta ^2 - c - k_1 k_3) \right ) E_1 (f) + \\
				& + \left ((4-r) \Omega + (r-m+3)\Theta \right ) \left ( -((r-1) \Omega + (m-r)\Theta ) E_1(f) +\right.\\
				& \left .+ (k_1 ^2 + (r-1) k_2^2 + (m-r) k_3^2 - mc)f\right)+\\
				& + \frac {9m^2(m-r+6)} {4(m-r)} f^2E_1(f) - \frac {3m(m+4r-2)} {m-r} fk_2 E_1(f) -\\
				& - \frac {3m(m+4r-2)} {2(m-r)} f^2E_1(k_2) + \frac {3m(r-1)} {m-r} k_2 ^2E_1(f)+\\
				& + \frac {6m(r-1)} {m-r} fk_2 E_1(k_2) - 3(m+1) c E_1(f) = 0,
			\end{align*}
			which is equivalent to 
			\begin{align}
				& \left ( (4-r) \left ( \frac m2 fk_2 - \Omega ^2 -c\right ) + (r-m+3) \left (\frac m2 fk_3 - \Theta ^2 -c \right ) \right ) E_1(f) - \label{Relation5}\\
				& - \left ( (4-r) \Omega + (r-m+3) \Theta \right ) \left ( (r-1) \Omega + (m-r) \Theta \right ) E_1(f) +\notag\\
				& + \left ( (4-r) \Omega + (r-m+3) \Theta \right ) \left ( \frac {m^2} 4 f^3 + (r-1) fk_2^2 + (m-r) fk_3^2 - mcf\right )+\notag\\
				& + \left ( \frac {9m^2(m-r+6)} {4(m-r)} f^2 - \frac {3m(m+4r-2)} {m-r} fk_2 + \frac {3m(r-1)} {m-r} k_2^2 - 3(m+1) c \right ) E_1(f) + \notag\\
				& + \left ( - \frac {3m(m+4r-2)} {2(m - r)} f^2 + \frac {6m(r-1)} {m-r} fk_2 \right ) E_1 (k_2) = 0.\notag
			\end{align}
			We will use (\ref{DefinitionOmegaTheta}), (\ref{Relation1}), (\ref{Relation3}), (\ref{Relation4}) in (\ref{Relation5}) to simplify it.
			\begin{align*}
				\ \ & -(4-r) \Omega ^2 E_1(f) - (r-m+3) \Theta ^2 E_1(f) + \\
				& +\left ( (4-r) \left ( \frac m2 f k_2 - c\right ) + (r-m+3) \left ( \frac {3m^2} {4(m-r)} f^2 - \frac {m(r-1)} {2(m-r)} fk_2 \right ) - c \right ) E_1(f) +\\
				& +\left ( (r-1)\Omega + (m-r)\Theta \right ) \left ( \frac {3m^2(m-r+6)} {4(m-r)} f^3 - \frac {3m(m+4r -2)} {2(m-r)} f^2k_2 + \right .\\
				& \left . + \frac {3m(r-1)} {m-r} f k_2 ^2 - 3 (m+1) cf \right ) +\\
				& + \left ((4-r) \Omega + (r-m+3)\Theta \right ) \left ( \frac {m^2(m-r+9)} {4(m-r)} f^3 + \frac {(r-1) (m-1)} {m-r} fk_2^2 +\right . \\
				& \left . + \frac {3m(1-r)} {m-r} f^2k_2 - mcf\right ) +\\
				& + \left ( \frac {9m^2(m-r+6)} {4(m-r)} f^2 - \frac {3m(m+4r -2)}{m-r} fk_2 + \frac {3m(r-1)} {m-r} k_2 ^2 - 3 (m+1)c \right ) E_1(f) + \\
				& + \left (-\frac {3m (m+4r-2)} {2(m-r)} f^2 + \frac {6m(r-1)} {m-r} fk_2 \right ) E_1(k_2) = 0 \\
				\Leftrightarrow& (m-r-3) \left ( c + \frac {3m} {2(m-r)} fk_2 - \frac {r-1} {m-r} k_2^2 \right ) E_1(f) + \left ( \frac {3m^2(m-r+6)} {4(m-r)} f^3 -\right .\\
				& - \left . \frac {3m(m+4r-2)} {2(m-r)} f^2 k_2 + \frac {3m (r-1)} {m-r} fk_2 ^2 - 3 (m+1) cf\right ) \Omega +\\
				& + (r - 4) \left ( c + \frac {3m} {2 (m-r)} fk_2 - \frac {r-1} {m-r} k_2 ^2 \right ) E_1 (f) + \left ( \frac {3m^2(m-r+6)} {4(m-r)} f^3 - \right .\\
				& - \left . \frac {3m(m+4r -2)} {2(m-r)} f^2 k_2 + \frac {3m(r-1)} {m-r} fk_2 ^2 - 3(m+1) cf \right ) \Theta + \\
				& + \left ( \frac {3m^2(r-m+3)} {4(m-r)} f^2 + \frac {3m(m-2r+1)} {2(m-r)} fk_2 + (m-7) c \right ) E_1(f) +\\
				& + \left ( (r-1) \Omega + (m-r)\Theta \right ) \left ( \frac {3m^2(m-r+6)} {4(m-r)} f^3 - \frac {3m(m+4r-2)} {2(m-r)} f^2k_2 +\right . \\
				& \left . + \frac {3m(r-1)} {m-r} fk_2^2  - 3 (m+1)c f\right ) + \\
				& +\left ( (4-r) \Omega + (r-m+3)\Theta \right ) \left ( \frac {m^2(m-r +9)} {4(m-r)} f^3 + \frac {(r-1) (m-1)} {m-r} fk_2 ^2 + \right .\\
				& \left . + \frac {3m(1-r)} {m-r} f^2k_2 - mcf \right ) +\\
				& + \left ( \frac {9m^2(m-r+6)} {4(m-r)} f^2 - \frac {3m(m+4r-2)} {m-r} fk_2 + \frac {3m(r-1)} {m-r} k_2 ^2 - 3(m+1) c\right ) E_1(f) +\\
				& + \left ( - \frac {3m(m+4r-2)} {2(m-r)} f^2 + \frac {6m(r-1)} {m-r} fk_2 \right ) \Omega \cdot (k_1 - k_2) = 0\\
				\Leftrightarrow & (m-7) \left (c + \frac {3m} {2(m-r)} fk _2 - \frac {r-1} {m-r} k_2^2\right ) E_1(f) + \\
				& + \left (\frac {3m^2(m-r+6)} {4(m-r)} f^3 - \frac { 3m (m+4r-2)} {2(m-r)} f^2k_2 + \frac {3m(r-1)} {m-r} fk_2^2 - 3(m+1)cf \right ) \Omega +\\
				& + \left (\frac {3m^2(m-r+6)} {4(m-r)} f^3 - \frac { 3m (m+4r-2)} {2(m-r)} f^2k_2 + \frac {3m(r-1)} {m-r} fk_2^2 - 3(m+1)cf \right ) \Theta +\\
				& + \left ( \frac {3m^2(r-m+3)} {4(m-r)} f^2 + \frac {3m(m-2r+1)} {2(m-r)} fk_2 + (m-7) c\right ) E_1(f) + \\
				& + \left ( (r-1) \Omega + (m-r) \Theta \right ) \left ( \frac {3m^2(m-r+6)} {4(m-r)} f^3 - \frac {3m(m+4r -2)} {2(m-r)} f^2k_2 +\right .\\
				& + \left . \frac {3m(r-1)} {m-r} fk_2^2 - 3(m+1) cf\right )+\\
				& + \left ( (4-r) \Omega + (r-m+3) \Theta \right ) \left ( \frac {m^2(m-r+9)} {4(m-r)} f^3 + \frac {(r-1)(m-1)} {m-r} fk_2^2 + \right . \\
				& + \left . \frac {3m(1-r)} {m-r} f^2k_2 - mcf\right ) + \\
				& + \left ( \frac {9m^2(m-r+6)} {4(m-r)} f^2 - \frac {3m(m+4r -2)} {m-r} fk_2 + \frac {3m(r-1)} {m-r} k_2^2 - 3(m+1) c \right ) E_1(f) + \\
				& + \left ( \frac {3m^2(m+4r-2)} {4(m-r)} f^3 + \frac {3m(3m-2mr+4r-2)} {2(m-r)} f^2k_2 + \frac {6m(1-r)} {m-r} fk_2^2\right ) \Omega = 0.
			\end{align*}
			This relation is equivalent to
			\begin{align}
				& \left ( \frac {3m^2(2m -2r +21)} {4(m-r)} f^2 - \frac {3m(1+5r)} {m-r} fk_2 + \frac {(r-1) (7+2m)} {m-r} k_2 ^2 - (m+17)c \right ) E_1(f) +\label{FirstDifference}\\
				& + \left ( \frac {m^2(7m+17r+30+2mr-2r^2)}{4(m-r)} f^3 + \frac {3m(6-4r+3m-3mr-2r^2)}{2(m-r)} f^2k_2 +\right . \notag\\
				& + \left . \frac {(r-1)(2mr-2m+r-4)}{m-r} fk_2^2 - (4m + 2mr + 3r) cf \right ) \Omega + \notag\\
				& + \left ( \frac {m^2(15m-15r+45+2m^2-4mr+2r^2)} {4(m-r)} f^3 + \frac {3m(8-10r-m-mr+2r^2-m^2)} {2(m-r)} f^2k_2 + \right . \notag\\
				& + \left . \frac {(r-1)(7m+2m^2-2mr-r-3)} {m-r} fk_2^2 + (2mr-9m-3+3r-2m^2)cf \right ) \Theta = 0. \notag
			\end{align}
			Using (\ref{RelationDer8}) the last relation becomes (\ref{Relation6}).
		\end{proof}
		
		\begin{remark}
			Until relation (\ref{Relation6}), see also \eqref{FirstDifference}, all our computations coincide with the Yu Fu's computations. Our relation \eqref{FirstDifference} slightly differs from relation just below (3.49) in \cite{YuFu}.
		\end{remark}
		Finally, from Lemma \ref{LemmaPropertiesOmegaTheta} we can deduce the expression of our first polynomial equation in $f$ and $k_2$ mentioned at the beginning of the proof.
		 
		First, we will denote by $P$ and $Q$ the coefficients of $\Omega$ and $\Theta$, respectively, in (\ref{Relation6}) and by $R$ the right-hand side of (\ref{Relation2}). Thus, we get 
		\begin{equation}\label{Relation7}
			(r-1) (4-r) (mf+2k_2) \Omega ^2 + (3+r-m) ( m(m-r+3) f - 2(r-1) k_2) \Theta ^2 = R,
		\end{equation}
		\begin{equation}\label{Relation8}
			P\Omega + Q\Theta = 0.	
		\end{equation}
		Multiplying (\ref{Relation7}) by $PQ$ we obtain
		\begin{align}
			& (3+r-m) \left ( m(m-r+3) f - 2(r-1) k_2\right ) P^2 \left ( c + \frac {3m} {2(m-r)} fk_2 - \frac {r-1} {m-r} k_2^2\right ) + \label{Relation9}\\
			& + (r-1)(4-r) (mf+2k_2) Q^2 \left (c + \frac {3m} {2(m-r)} fk_2 - \frac {r-1} {m-r} k_2^2\right ) = PQR. \notag
		\end{align}
		Equation (\ref{Relation9}) can be written as 
		\begin{equation}\label{Polynomial1}
			\sum _{i = 0} ^9 a_{i,9-i} k_2^i f^{9-i} + c\left ( \sum _{i = 0} ^7 a_{i,7-i} k_2^i f^{7-i} + \sum _{i = 0} ^5 a_{i,5-i} k_2^i f^{5-i} + \sum _{i = 0} ^3 a_{i,3-i} k_2^i f^{3-i} \right ) = 0,
		\end{equation}
		where the coefficients $a_{ij}$ depend on $m,\ r$ and $c$, thus they are constants.
		
		We can easily show that for any $m$, $r$ and $c$
		\begin{align*}
			& a_{9,0} = \frac {729m^9(2m-2r+3)(3m-2r+17)(m-r+6)} {32(m-r)} > 0,\\
			& a_{0,9} = 0.
		\end{align*}
		Therefore, the left hand-side of (\ref{Polynomial1}) is a non-zero polynomial.
		
		If $k_2$ is constant on $M$ then, from (\ref{Polynomial1}), we obtain a $9^{th}$-degree polynomial in the variable $f$ with constant coefficients, thus $f$ is constant on $M$, contradiction.
		
		We will assume that $k_2$ is not constant on $M$. Restricting $M$, if necessary, we can suppose that $\grad k_2 \neq 0$ and $k_2 \neq 0$ at any point of $M$. The fact that $\grad k_2 = E_1(k_2) E_1 \neq 0$ at any point of $M$ implies that $E_1(k_2) \neq 0$ at any point of $M$. Therefore $\Omega \neq 0$ at any point of $M$.
		
		Let $\gamma : I \to M$ be an integral curve of $E_1$, $\gamma = \gamma(t)$. 
		
		\begin{lemma}\label{LemmaKFconst}
			Along $\gamma$, the ratio $k_2/f$ cannot be constant.
		\end{lemma}
		\begin{proof}
			We will work on $\gamma$. Assume, by way of contradiction, that $k_2/f$ can be a constant. Let $\alpha \in \mathbb R$ be a non-zero constant such that $k_2 = \alpha f$. We have
			$$
			k_3 = \frac {3m - 2\alpha (r-1)} {2(m-r)} f.
			$$
			Thus, 
			$$
				k_1 = -\frac m2 f,\quad k_2 = \alpha f,\quad k_3 = \beta f,
			$$
			where $\beta = \frac {3m - 2\alpha (r-1)} {2(m-r)}$. Using \eqref{CurvaturesNEQ} we get that $\alpha \neq -m/2$, $\beta \neq - m/2$ and $\alpha \neq \beta$.
			
			From \eqref{EquationGauss1} we obtain 
			\begin{equation}\label{EquationKFconst1}
				fE_1(E_1(f)) = \frac {m+4\alpha} {m+2\alpha} \left (E_1(f) \right )^2 + \frac {m+2\alpha} {2\alpha} c f^2 - \frac {m(m+2\alpha)} 4 f^4.
			\end{equation}
			We want to show that $\beta \neq 0$. If $\beta = 0$, then $k_3 = 0$ and relation \eqref{EquationGauss2} is equivalent to $c = 0$.
			
			Using relation \eqref{EquationNormalPart} we get 
			\begin{equation}\label{EquationKFconst3}
				fE_1\left ( E_1(f)\right ) = - \alpha (r-1) f \left ( E_1(f)\right )^2 + \frac {m^2+4(r-1)\alpha ^2} 4 f^4.
			\end{equation}
			Relation \eqref{EquationKFconst1} becomes
			\begin{equation} \label{EquationKFconst4}
				fE_1\left ( E_1(f)\right ) = \frac {m+4\alpha} {m+2\alpha} \left ( E_1(f) \right ) ^2 - \frac {m(m+2\alpha)} 4 f^4.
			\end{equation}
			We consider 
			$$
			w(t) = \left ( E_1(f) \right )^2(\gamma (t)) = \left ( \gamma '(t) (f)\right )^2 = \left ( \left ( f \circ \gamma \right )' (t)\right )^2 = \left ( f'(t)\right ) ^2.
			$$
			Since $t \longmapsto f(t)$ is a diffeomorphism and $t = t(f)$, we have $w = w(t) = w \left ( t (f)\right )$, $f \in f(I)$. We denote by $\overline w (f) = w (t(f))$ and obtain 
			\begin{align*}
				\frac {d\overline w} {df} (f) =& \frac {dw} {dt} \left ( t (f)\right ) \frac {dt} {df} (f) = \frac {dw} {dt} (t(f)) \frac 1 {\frac {df} {dt} \left (t(f)\right )}\\
											  =& 2 f''(t) f'(t) \frac 1 {f'(t)} = 2 f''(t) = 2 \left ( E_1\left (E_1(f)\right ) \right )	(\gamma (t)) .
			\end{align*}
			Relations \eqref{EquationKFconst3} and \eqref{EquationKFconst4} become
			\begin{equation*}
				\left \{
				\begin{array}{rl}
					\displaystyle \frac 12 f \frac {d \overline w} {df} =&\displaystyle \frac {m^2 + 4(r-1)\alpha ^2} 4 f^4 - \alpha (r-1) f \overline w\vspace{5pt}\\
					\displaystyle \frac 12 f \frac {d \overline w} {df} =&\displaystyle - \frac {m(m+2\alpha)} 4 f^4 + \frac {m+4\alpha} {m+2\alpha} \overline w
				\end{array}
				\right ..
			\end{equation*}
			Using these two equations we obtain that
			\begin{align*}
				\overline w = \frac {\left (m^2+4(r-1)\alpha ^2 + m(m+2\alpha)\right )(m+2\alpha)} 4 \cdot \frac {f^4} {(m+2\alpha) \alpha (r-1) f + m + 4\alpha}.
			\end{align*}
			Differentiating with respect to $f$ we get
			\begin{align*}
				\frac {d \overline w} {df} =& \frac {\left (m^2+4(r-1)\alpha ^2 + m(m+2\alpha)\right )(m+2\alpha)} 4 \cdot\\
										    & \cdot \frac {4\left ( (m+2\alpha) \alpha (r-1) f + m+4\alpha\right ) f^3 - (m+2\alpha)\alpha (r-1) f^4} {\left ( (m+2\alpha) \alpha (r-1) f + m + 4\alpha\right )^2}.
			\end{align*}
			Substituting in the second equation of the system we get 
			\begin{align*}
				& \frac {m(m+2\alpha)} 4 f^4 \left ( (m+2\alpha)\alpha(r-1)f + m+4\alpha\right )^2 +\\
				& + \frac 1 8 f \left ( m^2 + r(r-1)\alpha ^2 + m(m+2\alpha) \right )(m+2\alpha)\cdot\\
				& \cdot \left ( 3(m+2\alpha) \alpha (r-1) f^4 + 4(m+4\alpha)f^3\right ) -\\
				& - \frac {(m+4\alpha)\left (m^2 + 4(r-1)\alpha^2+m(m+2\alpha)\right )} 4 f^4 \cdot\\
				& \cdot \left ((m+2\alpha)\alpha(r-1)f+m+4\alpha\right )=0
			\end{align*}
			We obtain a $6^{th}$-degree polynomial relation in $f$ with the dominant term 
			$$
			\frac 14 m(m+2\alpha)^3\alpha^2(r-1)^2.
			$$
			Since $\alpha \neq -m/2$, we obtain a polynomial equation in the variable $f$, with constant coefficients, thus $f$ is constant along $\gamma$, contradiction.
			
			Therefore, $\beta \neq 0$.
			
			From \eqref{EquationGauss2} we get
			\begin{equation}\label{EquationKFconst5}
				fE_1\left ( E_1 (f)\right ) = \frac {m+4\beta} {m+2\beta} \left ( E_1(f)\right )^2 + \frac {m+2\beta} {2\beta} cf - \frac {m(m+2\beta)} 4 f^4.
			\end{equation}
			Using \eqref{EquationGauss3} we get
			\begin{equation}\label{EquationKFconst2}
				\left ( E_1(f)\right ) ^2 = -\frac {(m+2\alpha) (m+2\beta)} {4\alpha \beta} cf^2 - \frac {(m+2\alpha)(m+2\beta)} 4 f^4.
			\end{equation}
			Combining \eqref{EquationKFconst1}, \eqref{EquationKFconst2} and \eqref{EquationKFconst5} we obtain
			$$
			\frac {m(\beta - \alpha)} {\alpha \beta} cf^2 + m(\beta - \alpha)f^4 = 0.
			$$
			Since $\beta \neq \alpha$, the last relation is a polynomial equation with constant coefficients in the variable $f$, contradiction.
		\end{proof}
		
		We can distinguish two cases.
		
		\begin{center}
			\textbf{Case 1: } $\mathbf{c = 0}$.
		\end{center}
		Relation \eqref{Polynomial1} becomes
		$$
		\sum _{i = 0} ^9 a_{i,9-i} k_2^i f^{9-i} = 0.
		$$
		Using the fact that $f > 0$, we can divide with $f^9$ and obtain
		$$
		\sum _{i = 0} ^9 a_{i,9-i} \left ( \frac {k_2} f\right )^i = 0.
		$$
		Therefore, we get a polynomial equation in the variable $z = k_2 / f$ with constant coefficients. We have seen that $a_{9,0}$ is not zero and does not depend on $c$, thus this polynomial is non-zero and this implies that $z$ is constant, which contradicts Lemma \ref{LemmaKFconst}.
		
		\begin{center}
			\textbf{Case 2:} $\mathbf {c \neq 0}$.
		\end{center}
		Along $\gamma$ we have
		$$
		k_2'(t)= \left ( E_1(k_2) \right ) (\gamma (t)) \neq 0, \forall t \in I,
		$$
		where $k_2 := k_2 \circ \gamma$. In this case $t \longmapsto k_2(t)$ is a diffeomorphism and $t = t(k_2)$. 
		
		Next, we will denote $$f (t) = (f\circ \gamma) (t), \quad \tilde f (k_2) = f(t(k_2))\quad \text{and} \quad \tilde \gamma (k_2) = \gamma (t (k_2)).$$
		
		A direct consequence of Lemma \ref{LemmaKFconst} is
		\begin{remark}\label{RemarkKFconst}
			Along $\tilde \gamma$, the ratio $k_2/\tilde f$ cannot be a constant.
		\end{remark}
		
		\begin{lemma}\label{LemmaQnotVanishes}
			Along $\tilde \gamma$, $\tilde Q = Q \circ \tilde \gamma$ does not vanishes.
		\end{lemma}
		\begin{proof}
			If $\widetilde Q (k_2) = 0$, for any $k_2$, then, using the fact that $\widetilde \Omega \neq 0$, from \eqref{Relation8} we obtain $\widetilde P = 0$.
			
			The relations $\widetilde P (k_2) = 0$ and $\widetilde Q (k_2) = 0$, for any $k_2$, can be thought of as two polynomial equations in $k_2$ with coefficients depending on the function $\tilde f = \tilde f (k_2)$. We arbitrarily set $k_2 = k_2^0$ and thus the coefficients of the above two equations become constants. Now, we consider two polynomial equations in the variable $z = k_2$ with the corresponding above constant coefficients. Clearly, $z = k_2^0$ is a common solution of the last two equations. Therefore, the resultant of these polynomials with constant coefficients has to be $0$. The resultant, which is a real number, can be written as a polynomial relation in $\tilde f = \tilde f (k_2^0)$. Letting $k_2^0$ free, we get that $\tilde f = \tilde f (k_2)$ is a solution of a polynomial equation with constant coefficients. Using Mathematica (see Appendix \ref{AppCasePis0}), it can be shown that, since $c\neq 0$, this polynomial is non-zero. The fact that $\tilde f$ is continuous implies that $\tilde f$ is a constant function, thus $f$ is constant along $\tilde \gamma$, contradiction.
		\end{proof}
		
		We can express the derivative of $\tilde f$ with respect to $k_2$ as a rational relation in $\tilde f$ and $k_2$.
		\begin{lemma}\label{LemmaDerivativeF}
			Along $\tilde \gamma$ the derivative of $\tilde f$ with respect to $k_2$ is 
			\begin{equation}\label{DerF}
				\frac {d\tilde f} {dk_2} = \frac {2(r-1)} {3m} - \frac {2(m(m-r+3)\tilde f - 2(r-1) k_2 ) \widetilde P} { 3m(m\tilde f + 2k_2) \widetilde Q}.
			\end{equation}
		\end{lemma}
		\begin{proof}
			We have
			\begin{align*}
				\frac {d\tilde f} {dk_2} (k_2) =& \frac {df} {dt} (t(k_2)) \frac {dt} {dk_2} (k_2) = \frac {df} {dt} (t(k_2)) \frac 1 {\frac {dk_2} {dt} (t(k_2))} \\
				=& \left ( \frac {E_1(f)} {E_1(k_2)} \right ) \left ( \tilde \gamma (k_2) \right ).
			\end{align*}
			Therefore, along $\tilde \gamma $ we have
			\begin{align*}
				\frac {d\tilde f} {dk_2} (k_2) =& \left (\frac {E_1(f)} {E_1(k_2)} \right ) (\tilde \gamma(k_2))\\
				=& \frac { - \left ( \frac {r-1} 3 \tilde f + \frac {2(r-1)} {3m} k_2 \right ) \widetilde \Omega + \left ( - \frac {m - r + 3} 3 \tilde f + \frac {2(r-1)} {3m} k_2 \right ) \widetilde \Theta} { \left (-\frac m2 \tilde f - k_2\right ) \widetilde \Omega}.
			\end{align*}
			Using (\ref{Relation8}) we obtain the conclusion.
			\end{proof}
			
			In light of Lemma \ref{LemmaQnotVanishes}, we can restrict $\tilde \gamma$, if necessary, and assume that $\widetilde Q(k_2) \neq 0$ for any $k_2$.
			
			To get the second polynomial equation, we will differentiate the first polynomial given by \eqref{Polynomial1} with respect to $k_2$ along $\tilde \gamma$ and substituting the derivative of $\tilde f$ with respect to $k_2$ from Lemma \ref{LemmaDerivativeF}, we obtain
			\begin{align}
				\sum _{i=0} ^{12} b_{i,12-i} k_2^i \tilde f ^{12 - i} +& c\left ( \sum _{i=0} ^{10} b_{i,10-i} k_2^i \tilde f ^{10 - i} + \sum _{i=0} ^{8} b_{i,8-i} k_2^i \tilde f ^{8 - i} + \right . \label{Polynomial2}\\
				& + \left . \sum _{i=0} ^{6} b_{i,6-i} k_2^i \tilde f ^{6 - i} + \sum _{i=0} ^{4} b_{i,4-i} k_2^i \tilde f ^{4 - i} \right ) = 0,\notag
			\end{align}
			where the coefficients $b_{ij}$ depend on $m$, $r$ and $c$ and are constants.
			
			Relations \eqref{Polynomial1} and \eqref{Polynomial2} can be seen as two polynomial equations in $k_2$ with coefficients depending on the function $\tilde f = \tilde f(k_2)$. As in the proof of Lemma \ref{LemmaQnotVanishes}, we can compute the resultant of the two polynomials and, finally, we can obtain a polynomial in $\tilde f$ with constant coefficients. If this polynomial is non-zero, then we get a contradiction and we end the proof.
			
			Since the derivative of $\tilde f$ with respect to $k_2$ is different from zero, for any $k_2$, we can change the point of view and \eqref{Polynomial1} together with \eqref{Polynomial2} can be thought of as two polynomial equations in $\tilde f$ with coefficients depending on the function $\tilde k_2 = \tilde k_2 (\tilde f)$. As we described above, we can compute the resultant for these new polynomials obtaining a polynomial in $\tilde k_2$ with constant coefficients. Again, if this polynomial is non-zero, then we get a contradiction.
			
			Since the volume of computations is very big, we could not compute the resultant for generic $c$, $m$ and $r$. Because of that, we made a programme in Mathematica which computes the resultant for any particular choice of $c$, $m$ and $r$.
			
			In the first situation, when the resultant is a polynomial in $\tilde f$, we obtain that, for any $c \neq 0$, $m \in \overline {4,30}$ and all possible values of $r$, the only case when the resultant is the zero polynomial is given by
			$$
			m=7 \quad \text{and} \quad r=4
			$$
			(see Appendix \ref{AppResultantF}).
			
			In the second situation, when the resultant is a polynomial in $k_2$, the resultant (which should be a polynomial in $k_2$) is the zero polynomial for any $c\neq 0$, $m \in \overline {4, 30}$ and for all possible values of $r$ (see Appendix \ref{AppResultantF}).
			
			In order to reduce the volume of computations and to compute the resultant for the generic case, we will reduce the degree of polynomials in \eqref{Polynomial1} and \eqref{Polynomial2}. Thus, we divide \eqref{Polynomial1} by $\tilde f^3$ and \eqref{Polynomial2} by $\tilde f^4$, and denoting $\tilde z = k_2/\tilde f$, we obtain 
			\begin{equation}\label{Polynomial3}
				\sum _{i = 0} ^9 a_{i,9-i} \tilde z^i \tilde f^{6} + c\left ( \sum _{i = 0} ^7 a_{i,7-i} \tilde z^i \tilde f^{4} + \sum _{i = 0} ^5 a_{i,5-i} \tilde z^i \tilde f^{2} + \sum _{i = 0} ^3 a_{i,3-i} \tilde z^i \right ) = 0
			\end{equation}
			and
			\begin{align}
				\sum _{i=0} ^{12} b_{i,12-i} \tilde z^i \tilde f ^{8} +& c\left ( \sum _{i=0} ^{10} b_{i,10-i} \tilde z^i \tilde f ^{6} + \sum _{i=0} ^{8} b_{i,8-i} \tilde z^i \tilde f ^{4} + \right . \label{Polynomial4}\\
				& + \left . \sum _{i=0} ^{6} b_{i,6-i} \tilde z^i \tilde f ^{2} + \sum _{i=0} ^{4} b_{i,4-i} \tilde z^i \right ) = 0.\notag
			\end{align}
			
			Using Mathematica we can compute the resultant of these polynomials in the general case and prove that it vanishes only when $m=7$ and $r=4$ (see Appendix \ref{AppNewPolynomials}).
			
			\begin{center}
				\textbf{Case 3: $\mathbf{c\neq 0}$, $\mathbf{m=7}$ \textbf{and} $\mathbf{r=4}$.}
			\end{center}
				
			In this case equations \eqref{Polynomial1} and \eqref{Polynomial2} become
			\begin{align}
				& \frac {45927}{16} \left (32 c - 147 \tilde f^2 + 98 \tilde f k_2 - 28 k_2^2\right )\cdot \label{PolynomialSpecialCase1}\\
				& \cdot \left (336 c \tilde f - 1715 \tilde f^3 - 32 c k_2 + 1470 \tilde f^2 k_2 - 294 \tilde f k_2^2 + 28 k_2^3\right )\cdot \notag \\
				& \cdot \left (32 c (7 \tilde f + k_2) + 7 \left (147 \tilde f^3 - 63 \tilde f^2 k_2 - 4 k_2^3\right )\right ) = 0\notag
			\end{align}
			and
			\begin{align}
				& \frac {1240029}{16} \left (7 \tilde f - 4 k_2\right ) \left (32 c - 147 \tilde f^2 + 98 \tilde f k_2 - 28 k_2^2\right )\cdot \label{PolynomialSpecialCase2}\\
				& \cdot \left [81920 c^3 \left (343 \tilde f^2 + 7 \tilde f k_2 - 2 k_2^2\right ) - 7168 c^2 \left (66542 \tilde f^4 - \right . \right.\notag \\
				& - \left . 12201 \tilde f^3 k_2 + 2653 \tilde f^2 k_2^2 + 476 \tilde f k_2^3 - 68 k_2^4\right ) - \notag\\
				& - 784 c \left (2384193 \tilde f^6 - 1172717 \tilde f^5 k_2 + 559384 \tilde f^4 k_2^2 - \right .\notag \\
				& - \left . 154252 \tilde f^3 k_2^3 + 40656 \tilde f^2 k_2^4 - 6384 \tilde f k_2^5 + 608 k_2^6\right ) + \notag\\ 
				& + 2401 \left (6950895 \tilde f^8 - 10169607 \tilde f^7 k_2 + 5436942 \tilde f^6 k_2^2 -\right . \notag \\
				& - 1685894 \tilde f^5 k_2^3 + 421456 \tilde f^4 k_2^4 - 69608 \tilde f^3 k_2^5 + \notag\\
				& + \left .\left . 10288 \tilde f^2 k_2^6 - 896 \tilde f k_2^7 + 64 k_2^8\right )\right ] = 0,\notag
			\end{align}
			respectively.
			
			We see from \eqref{PolynomialSpecialCase1} and \eqref{PolynomialSpecialCase2} that the conic $\delta = 28 k_2^2 - 98 \tilde f k_2 + 147 \tilde f^2 - 32 c$ is the only common factor. The conic is an ellipse when $c>0$ and an imaginary ellipse when $c<0$. 
			
			On the other hand, from \eqref{Relation1} we obtain that $\delta = 0$. Thus, as we already mentioned, this case is nothing but Lemma \ref{LemmaBicons}.
			
			We have $\delta (k_2) = 0$, for any $k_2$. In the following, we do not need to work with the variable $k_2$, so we will come back to the first variable $t$. Differentiating the relation $\delta (t) = 0$, we obtain
			$$
			7(3f(t) - k_2(t))f'(t)+(4k_2(t)-7f(t))k_2'(t) = 0.
			$$
			We recall that 
			$$
			k_1 = -\frac 72 f\quad \text{and} \quad k_3 = \frac 72 f - k_2.
			$$
			Using Lemma \ref{LemmaKFconst}, we get that
			$$
			k_2'(t) = \frac {7(k_2(t) - 3f(t))} {4k_2(t)-7f(t)} f'(t)
			$$
			and thus, we have
			$$
			k_3'(t) = \frac {7(2k_2(t)-f(t))}{2(4k_2(t) - 7f(t))}f'(t).
			$$
			We can write $\delta = 0$ as follows
			\begin{equation}\label{CanonicalFormDelta}
				7(4k_2-7f)^2=128c - 245f^2
			\end{equation}
			Using \eqref{DefinitionOmegaTheta}, we obtain
			\begin{align*}
				\Omega =& -\frac {14(k_2-3f)} {(7f+2k_2)(4k_2 - 7f)} f'\\
				\Theta =& \frac {7(2k-f)} {2(k_2 - 7f)(4k_2-7f)} f'.
			\end{align*}
			Thus, \eqref{EquationGauss3} can be written as
			\begin{equation}\label{EquationGauss3Particular1}
				98(k_2-3f)(2k_2-f)\left (f'\right )^2 = \left (7fk_2-2k_2^2+2c\right )\left (4k_2-7f\right )^2(7f+2k_2)(k_2-7f).
			\end{equation}
			Using the fact that $\delta = 0$ and its equivalent form \eqref{CanonicalFormDelta}, we obtain
			\begin{equation}\label{EquationGauss3Particular2}
				98^2\left (32c-105f^2\right ) \left ( f' \right )^2 = \left (147f^2 - 4c\right )\left (128c - 245f^2\right )\left (-833f^2+32c\right ).
			\end{equation}
			Differentiating \eqref{EquationGauss3Particular2} we get
			\begin{align}
				 & -20580 f \left (f'\right )^2 + 196\left ( 32 c - 105f^2\right ) f{''} = \label{EquationDerFParticular1}\\
				=& 3f(128c - 245 f^2) \left (-833f^2 + 32c\right )- \notag\\
				 & -5f\left ( 147 f^2-4c\right ) \left ( -833 f^2 + 32c\right ) - \notag\\
				 & -17 f \left (147 f^2 - 4c\right ) \left ( 128c - 245 f^2\right ). \notag
			\end{align}
			From \eqref{EquationNormalPart}, we obtain 
			\begin{equation}\label{EquationDerFParticular2}
				f{''} = \frac {1911 f} {833 f^2 - 32 c} \left (f'\right ) ^2 + \frac 1 {14} \left ( 245 f^2 - 2c\right ) f
			\end{equation}
			Substituting \eqref{EquationDerFParticular2} in \eqref{EquationDerFParticular1} and using \eqref{EquationGauss3Particular2} we obtain
			\begin{align*}
			 	 & 14386462720 \cdot c^4 f - 356598824960 \cdot c^3 f^3 - 2331746708480 \cdot c^2 f^5 +\\ 
			 	 & + 42758681977200 \cdot c f^7 + 151265495839500 f^9 = 0
			\end{align*}
			We get a $9^{th}$-degree polynomial in the variable $f$ with constant coefficients, which is a contradiction.
			
		\appendix
		\section{}\label{AppCasePis0}
		\noindent Here we present the Mathematica code for computing the resultant of the two polynomials that appear in the proof of Lemma \ref{LemmaQnotVanishes}.
		
		First, we have to declare $\widetilde P$ and $\widetilde Q$. For simplicity, we denote them by $P$ and $Q$, respectively. Also, $\tilde f$ is denoted by $f$ and $k_2$ by $k$.
		
		\begin{verbatim}
			P = 9/4 m^3*(3 m - 2 r + 17) f^3 + 
			3/2 m^2*(6 r^2 - 43 r + 37 + 11 m - 11 m*r)*f^2*k + 
			m*(r - 1)*(26 r + 4 m*r + 1 - 4 m)* f*k^2 + 
			m*(m - r) (8 r - 5 m*r - 13 m - 17) c*f -
			2 (r - 1)^2 (7 + 2 m)*k^3 + 2 (m - r)*(r - 1)*(m + 17)*c*k
		\end{verbatim}
	
		\begin{verbatim}
			Q = 9/2 m^3*(2 r - 2 m - 3)*f^3 + 9/2 m^2*(7 r - m + 3 - m^2 + 
			3 m*r - 2 r^2)*f^2*k + 2 m*(r - 1)*(4 m - 13 r - 18 - 2 m*r + 2 m^2)*f*k^2 +
			m*(m - r)*(5 m*r - 5 m^2 - 7 m - 8 r + 42)*c*f + 
			2 (r - 1)^2*(7 + 2 m)*k^3 - 2 (r - 1)*(m - r)*(m + 17)*c*k
		\end{verbatim}
		For simplicity, we denoted $k_2$ by $k$. 
		
		Now, we compute the resultant of $P$ and $Q$ with respect to $k$ and simplify it.
		
		\begin{verbatim}
			ResPQ[m_][r_][c_][f_] = Collect[Resultant[P, Q, k], f, FullSimplify]
		\end{verbatim}
		We want to prove that this resultant is not the zero polynomial. We obtain a $9^{th}$-degree polynomial in $f$, free of the constant term and with vanishing coefficients of $f$ and $f^2$. We will look to the coefficient of $f^3$ since it is the first non-zero monomial.
		\begin{verbatim}
			CoefF3 = Coefficient[ResPQ[m][r][c][f], f, 3]
		\end{verbatim}
		which yields
		$$
		1474560\cdot c^3 (-1 + m) m^3 (5 + m) (7 + 2 m)^3 \left (20 -3m + m^2\right )^2 (m - r)^3 (-1 + r)^6
		$$
		Since the integers $r$ and $m$ satisfy $1<r<m$, none of the factors of this coefficient can be zero, except $c^3$. 
		
		If $c\neq 0$, clearly the resultant is a non-zero polynomial.

		\section{}\label{AppResultantF}
		
		In addition of $\tilde P$ and $\tilde Q$ from the Appendix \ref{AppCasePis0}, we have to declare $\tilde R$. Again, for simplicity, we will denote $R = \tilde R$.
		
		\begin{verbatim}
			R = (9 m^3*(m - r + 6))/(4 (m - r))*
			f^3 + (3 m^2*(r - 1)*(2 r - 2 m - 15))/(2 (m - r))*f^2*
			k + (m*(r - 1)*(m + 11 r - 12 + 2 m*r - 2 r^2))/(m - r)*f*
			k^2 + (2 (r - 1)^2*(m - 2 r + 1))/(m - r)*k^3 - 
			m*(2 m*r + 4 m - 2 r^2 + 5 r)*c*f - 2 (r - 1)*(m - 2 r + 1)*c*k
		\end{verbatim}
	
		We use relation \eqref{Relation9} to obtain the polynomial given by \eqref{Polynomial1}.
		
		\begin{verbatim}
			Rel362 = (3 + r - m)*(m (m - r + 3)*f - 2 (r - 1) k)*
			P^2*(c + (3 m)/(2 (m - r))*f*k - (r - 1)/(m - r)*k^2) + (r - 
			1) (4 - r) (m*f + 2 k)*
			Q^2*(c + (3 m)/(2 (m - r))*f*k - (r - 1)/(m - r)*k^2) - P*Q*R
		\end{verbatim}
	
		Next, we define a function $H$ which is just \eqref{Relation9} simplified.
		
		\begin{verbatim}
			H[m_][r_][c_][f_][k_] = Total[FullSimplify[MonomialList[Rel362, 
			{f, k}]]]
		\end{verbatim}
	
		We input the derivative of $\tilde f$, now $f$, with respect to $k_2$ which is denoted by $k$.
		
		\begin{verbatim}
			DerF = (2 (r - 1))/(3 m) - (2 (m (m - r + 3)*f - 2 (r - 1)*k)*
			P)/(3 m (m*f + 2 k)*Q)
		\end{verbatim}
		
		Now, we declare the numerator and the denominator of this relation, respectively.
		
		\begin{verbatim}
			NumDerF = Numerator[Together[DerF]]
		\end{verbatim}
	
		\begin{verbatim}
			DenDerF = Denominator[Together[DerF]]
		\end{verbatim}
	
		If we denote by $\overline H \left (k_2, \tilde f \right )$ the polynomial in relation \eqref{Polynomial1}, we have
		$$
		\frac {d\overline H} {dk_2} \left (k_2, \tilde f (k_2)\right ) = \frac {\partial \overline H} {\partial k_2} \left ( k_2, \tilde f (k_2)\right ) + \frac {\partial \overline H} {\partial \tilde f} \left (k_2, \tilde f (k_2)\right ) \frac {d\tilde f} {d k_2} (k_2).
		$$
		Actually, $\overline H$ is denoted by $H$ in the code.
		
		We multiply this relation with the denominator of the derivative of $\tilde f$ with respect to $k_2$ in order to obtain the polynomial in \eqref{Polynomial2}.
		
		\begin{verbatim}
			Rel370 = D[H[m][r][c][f][k], f]*NumDerF +
			D[H[m][r][c][f][k], k]*DenDerF
		\end{verbatim}
		
		We define the function $K$ to be the polynomial from \eqref{Polynomial2} after simplifications.
		
		\begin{verbatim}
			K[m_][r_][c_][f_][k_]=Total[FullSimplify[MonomialList[Rel370, {f, k}]]]
		\end{verbatim}
		
		Now, we compute the resultant for $H$ and $K$ with respect to $k$, for all $c \in \{ -1, 0, 1\}$, $m \in \overline {4, 30}$ and $r \in \overline {2, m-1}$. We know that, since the biharmonicity and minimality are invariant under homothetic transformations, we can assume that $c \in \{ -1, 0, 1\}$.
		
		\begin{verbatim}
			For[cc = -1, cc < 2, cc++,
			 For[mm = 4, mm < 31, mm++,
			  For[rr = 2, rr < mm, rr++,
			   res = Resultant[H[mm][rr][cc][f][k], K[mm][rr][cc][f][k], k];
			   Print["The resultant with respect to k for m = ", mm, ", r = ", 
			   rr, ", c = ", cc, " is \n", res];
			   If[res === 0, Print["Exception"], ];
			  ]
			 ]
			]
		\end{verbatim}
	
		From these computations we find out that the resultant is the zero polynomial only in the case $m = 7$ and $r = 4$.
		
		Further, we compute the resultant of $H$ and $K$ with respect to $f$, for any $c \in \{ -1, 0, 1\}$, $m \in \overline {4, 30}$ and $r \in \overline {2, m-1}$.
		
		\begin{verbatim}
			For[cc = -1, cc < 2, cc++,
			 For[mm = 4, mm < 31, mm++,
			  For[rr = 2, rr < mm, rr++,
			   res = Resultant[H[mm][rr][cc][f][k], K[mm][rr][cc][f][k], f];
			   Print["The resultant with respect to f for m = ", mm, ", r = ", 
			   rr, ", c = ", cc, " is \n", res];
			   If[res === 0, Print["Exception"], ];
			  ]
			 ]
			]
		\end{verbatim}
		
		In this situation, the resultant is the zero polynomial for any $c$, $m$ or $r$.
		
		\section{}\label{AppNewPolynomials}
		
		We need to find the coefficients of lower degree polynomials \eqref{Polynomial3} and \eqref{Polynomial4}. To do this, we will create two matrices with the entries being the coefficients of the polynomials from \eqref{Polynomial1} and \eqref{Polynomial2}, respectively.
		
		\begin{verbatim}
			CoefH = CoefficientList[H[m][r][c][f][k], {k, f}]
		\end{verbatim}
	
		\begin{verbatim}
			CoefK = CoefficientList[K[m][r][c][f][k], {k, f}]
		\end{verbatim}
		
		These matrices are made such that the element from the position $(i, j)$ is the coefficient of $k_2^{i-1} \tilde f^{j-1}$ from $H$ and $K$, respectively (see Appendix \ref{AppResultantF}). We look for a formula that links the elements of these matrices with $a_{i, s-i}$ and $b_{i, \overline s -i}$ from \eqref{Polynomial1} and \eqref{Polynomial2}, respectively, where $s \in \{ 3, 5, 7, 9\}$ and $\overline s \in \{ 4, 6, 8, 10, 12 \}$.
		
		We will study the case of $\overline H$ form Appendix \ref{AppResultantF}, the other one being similar.
		For simplicity, let $\left (A_{ij} \right ) _{i, j \in \overline {1, 10}}$ be a matrix given by \textit{CoefH}, thus 
		$$
		\overline H = \sum _{i = 1} ^{10} \sum _{j = 1} ^{10} A_{ij} k_2^{i-1}\tilde f^{j-1}
		$$
		and $s \in \{ 3, 5, 7, 9\}$.
		
		If $(i-1) + (j-1) = s$, then $j = s - i +2$ and $A_{ij} = A_{i, s-i+2} = a_{i-1, s-(i-1)}$.
		
		If $(i-1) + (j-1) \neq s$, for any $s \in \{ 3, 5, 7, 9\}$, then $A_{ij} = 0$.
		
		Therefore,
		\begin{align*}
			\overline H =& \sum _{i = 1} ^{10} \sum _{j = 1} ^{10} A_{ij} k_2^{i-1}\tilde f^{j-1}\\
						=& \sum _{s \in \{ 3, 5, 7, 9\}} \sum _{i=1} ^{s+1} A_{i, s-i+2} k_2^{i-1}\tilde f^{s-i+1}\\
						=& \sum _{s \in \{ 3, 5, 7, 9\}} \sum _{i=1} ^{s+1} a_{i-1, s-(i-1)} k_2^{i-1}\tilde f^{s-(i-1)}\\
						=& \sum _{s \in \{ 3, 5, 7, 9\}} \sum _{i=0} ^{s} a_{i, s-i} k_2^{i}\tilde f^{s-i}
		\end{align*} 
		Thus, 
		$$
		\overline H = \sum _{i = 1} ^{10} \sum _{j = 1} ^{10} A_{ij} \left ( \frac {k_2} {\tilde f} \right ) ^{i-1} \tilde f ^{i+j-2}
		$$
		and
		$$
		\frac 1 {f^3} \overline H = \sum _{i = 1} ^{10} \sum _{j = 1} ^{10} A_{ij} \left ( \frac {k_2} {\tilde f} \right ) ^{i-1} \tilde f ^{i+j-5}.
		$$
		Now, we write down a little code to obtain the polynomials in \eqref{Polynomial3} and \eqref{Polynomial4}, respectively.
		
		\begin{verbatim}
			AuxH = 0; For[i = 1, i <= 9, i++,
			 For[j = 1, j <= 10, j++,
			 AuxH = AuxH + CoefH[[i]][[j]]*f^(i + j - 5)*z^(i - 1); 
			 ]
			]
			newH [m_][r_][c_][f_][z_] = AuxH
		\end{verbatim}
		
		\begin{verbatim}
			AuxK = 0; For[i = 1, i <= 12, i++,
			 For[j = 1, j <= 13, j++,
			 AuxK = AuxK + CoefK[[i]][[j]]*f^(i + j - 6)*z^(i - 1); 
			 ]
			]
			newK[m_][r_][c_][f_][z_] = AuxK
		\end{verbatim}
		
		Using the same approach as in Appendix \ref{AppResultantF}, we can compute the resultants for $c \in \{ -1, 1\}$, $m \in \overline {4, 30}$ and $r \in \overline {2, m-1}$ of \textit{newH} and \textit{newK} with respect to $f$ and $z$. We get that both resultants vanish only when $m=7$ and $r=4$. Recall that, since $c \neq 0$ and the biharmonicity and harmonicity are invariant under homothetic transformations, we can assume $c \in \{ -1, 1\}$.
		
		First, we will compute the resultant of these new polynomials with respect to $f$ in the generic case.
		
		\begin{verbatim}
			res = Resultant[newH[m][r][c][f][z], newK[m][r][c][f][z], f]
		\end{verbatim}
	
		The resultant consists of a constant multiplied by a squared polynomial, denoted by \textit{resPoly}. Since we study in which case this resultant vanishes, we will consider only the polynomial \textit{resPoly}.
		
		\begin{verbatim}
			resFinal[m_][r_][c_][z_] = Total[ParallelMap[FullSimplify,
			 MonomialList[resPoly, z]]]
		\end{verbatim}
		
		We obtain a $40^{th}$-degree polynomial in the variable $z$ with constant coefficients depending on $c$, $m$ and $r$. The dominant coefficient of \textit{resPoly} is 
		\begin{verbatim}
			dominantCoef = Coefficient[resFinal[m][r][c][z], z, 40]
		\end{verbatim}
		which yields
		\begin{align*} 
			& -6917529027641081856\cdot c^{12} (-10 + m)^3 (-7 + m)^3 (-3 + m) (-1 + 
			m)^2 m^8 (5 + m)\cdot\\
			& \cdot (7 + 2 m)^5 (7 - 5 m + m^2) (20 - 3 m + m^2)^4 (-196 + 23 m + 11 m^2) \cdot \\
			& \cdot (-497 + 16 m + 49 m^2) (1 + m - 2 r) (m - r)^{12} (-1 + r)^{28}.
		\end{align*}
		Since $m \geq 4$ and $r \in \overline {2, m-1}$ are integers, also using the command \textit{IntegerQ}, it is easy to see that this coefficient vanishes if and only if
		$$
		m=7 \quad \text{or} \quad m=10 \quad \text{or} \quad m=2r-1.
		$$
		We have seen that this resultant does not vanish if $m=7$ and $r\neq 4$ or if $m=10$.
		
		If $m = 2r-1$, we substitute $m$ in the resultant above
		
		\begin{verbatim}
			resSpecial = resFinal[2r-1][r][c][z]
		\end{verbatim}
		
		We obtain a $39^{th}$-degree polynomial and its dominant coefficient is 
		
		\begin{verbatim}
			FullSimplify[Coefficient[resSpecial, z, 39]]
		\end{verbatim}
		which yields
		\begin{align*}
			& -56668397794435742564352\cdot c^{12} (11 - 2 r)^2 (-4 + r)^4 (-2 + r) (-1 + r)^{42} (2 + r)\cdot\\
			& \cdot (-1 + 2 r)^9 (5 + 4 r)^5 (13 - 14 r + 4 r^2) (12 - 5 r + 2 r^2)^4 (-116 - 41 r + 49 r^2) \cdot \\
			& \cdot (-2356 - 1035 r + 120 r^2 + 112 r^3)
		\end{align*}
		Since $r$ is an integer and using the command \textit{IntegerQ} in Mathematica, the only possibilities when this coefficient is zero are 
		$$r = 2 \quad \text{or} \quad r=4.$$
		
		If $r = 2$, then $m = 3 < 4$.
		
		If $r = 4$, then $m = 7$.
		
		Therefore, the only case in which the resultant vanishes is $m=7$ and $r=4$.
		
	\addcontentsline{toc}{section}{Bibliografie}
	\bibliographystyle{acm}
	
\end{document}